\documentclass[a4paper,leqno]{article}
\usepackage[centertags]{amsmath}
\usepackage{amsfonts}
\usepackage{amssymb}
\usepackage{amsthm}
\usepackage[active]{srcltx} 

\usepackage{color}

\newtheorem{Th}{Theorem}
\newtheorem{Main-Th}{Main Theorem}
\newtheorem{Cor}{Corollary}
\newtheorem{Pro}{Proposition}
\newtheorem{Lem}{Lemma}
\theoremstyle{definition}

\newtheorem{Conv}{Convention}

\newtheorem{Nots}{Notations}
\theoremstyle{remark}
\newtheorem{Rk}{Remark}
\newtheorem{Rks}{Remarks}
\newenvironment{Proof}{{\bf Proof:}}
{%
\mbox{}%
\nolinebreak%
\hfill%
\rule{2mm}{2mm}%
\medbreak%
\par%
}

\numberwithin{equation}{section}

\renewcommand{\thefootnote}{\fnsymbol{footnote}}
\title{\bf\large  Natural Ricci Solitons on tangent and unit tangent bundles }

\author{Mohamed Tahar Kadaoui Abbassi and Noura Amri}
\date{}
\begin{document}
\maketitle

\begin{abstract}
Considering pseudo-Riemannian $g$-natural metrics on tangent bundles, we prove that the condition of being Ricci soliton is hereditary in the sense that a Ricci soliton structure on the tangent bundle gives rise to a Ricci soliton structure on the base manifold. Restricting ourselves to some class of pseudo-Riemannian $g$-natural metrics, we show that the tangent bundle is a Ricci soliton if and only if the base manifold is flat and the potential vector field is a complete lift of a conformal vector field. We give then a classification of conformal vector fields on a flat Riemannian manifold. When unit tangent bundles over a constant curvature Riemannian manifold are endowed with pseudo-Riemannian Kaluza-Klein type metric, we give a classification of Ricci soliton structures whose potential vector fields are fiber-preserving, inferring the existence of some of them which are non Einstein.


\renewcommand{\thefootnote}{\arabic{footnote}}


\renewcommand{\thefootnote}{\arabic{footnote}}
\end{abstract}


\section*{Introduction}

Let $M$ be a smooth manifold of dimension $n \geq 2$. A \emph{Ricci soliton} on $M$ is a triple $(g,V,\lambda)$, where $g$ is a pseudo-Riemannian metric on $M$, $\textup{Ric}$ the associated Ricci tensor, $V$ a vector field (called the {\em potential vector field}) and $\lambda$ a real constant, satisfying the equation
\begin{equation}
    \textup{Ric}+\frac{1}{2}\mathcal{L}_Vg=\lambda g,
\end{equation}
$\mathcal{L}$ being the Lie derivative. If there is a $C^\infty$-function on $M$ such that  $\textup{Ric}+\nabla^2 f=\lambda g$, for some real constant $\lambda$, then $(M,g)$ is said \emph{gradient Ricci soliton}, where $\nabla$ is the Levi-Civita connection of $(M,g)$. A Ricci soliton is said to be either  {\em shrinking}, {\em steady}, or {\em expanding}, according on whether $\lambda$ is negative, zero, or positive, respectively.  Ricci solitons are a natural generalization of Einstein manifolds (an Einstein metric, together with a Killing vector field $V$, is a trivial Ricci soliton). Furthermore, they are the self-similar solutions to the Ricci flow.

Ricci solitons have been intensively studied in many contexts and from many points of view (we may refer to  \cite{Cao} and \cite{Cho-Kno} for details on the geometry of Ricci solitons). We are interested, in this paper, in the study of Ricci soliton structures in the framework of the geometry of tangent bundles. Indeed, {Metrics on tangent and unit tangent sphere bundles have been an important source of examples in Differential Geometry. In particular, {\em $g$-natural metrics}, which generalize the {\em Sasaki metric} and still arise in a natural way from the metric of the base manifold, have been intensively studied during the last decades. Unlike the Sasaki metric which shows a very rigid behaviour, the large class of $g$-natural metrics provides examples for several different interesting geometric properties (cf. \cite{Abb-Cal1}, \cite{ACP1}, \cite{Abb-Kow2}, \cite{Abb-Kow3}, \cite{Abb-Sar4}, \cite{CMM}, \cite{CP2}, \cite{CP3}, \cite{Per} and references therein).

In \cite{KAC}, the authors treated the problem of finding Ricci soliton structures on the unit tangent bundle of a Riemannian manifold, endowed with a pseudo-Riemannian {\em Kaluza-Klein type} metric, which is an interesting subclass of the class of $g$-natural metrics that shares with the Sasaki metric the property of preserving the orthogonality of horizontal and vertical distributions. They obtain a rigidity result in dimension three, showing that there are no nontrivial Ricci solitons among g-natural metrics of Kaluza–Klein type on the unit tangent sphere bundle of any Riemannian surface. On the other hand, they proved that, while Ricci solitons determined by tangential lifts remain trivial (i.e. an Einstein manifold) in arbitrary dimension, horizontal lifts of vector fields related to the geometry of flat base manifold (namely, homothetic vector fields) produce nontrivial Ricci solitons metrics of Kaluza–Klein type. They also gave a complete characterization of Gradient Ricci solitons of Kaluza–Klein type.

In this paper, we are interested in \emph{natural Ricci soliton} structures on tangent and unit tangent bundles of Riemannian manifolds, i.e. those associated with pseudo-Riemannian $g$-natural metrics. Our purpose is twofold. On one hand, we investigate natural Ricci soliton structures on tangent bundles of Riemannian manifolds. We prove that every natural Ricci soliton structure on the tangent bundle gives rise to a Ricci soliton structure on the base manifold, confirming the ``heridity" phenomenon of $g$-natural metrics (cf. \cite{Abb-Sar4}). Restricting ourselves $g$-natural metrics which are linear combinations of the three classical lifts (Sasaki, horizontal and vertical) of the base metric with constant coefficients, we give a complete characterization of Ricci soliton structures on the tangent bundle. We prove, in particular, that the existence of such structures requires the flatness of the base manifold, which constitutes a kind of rigidity of such metrics. Furthermore, this ensures the existence of non-trivial natural Ricci soliton structures on the tangent bundle of a Riemannian manifold.

On the other hand, we are looking for nontrivial natural Ricci solitons structures on the unit tangent bundle when the base manifold is of constant sectional curvature which is not necessarily zero. In this sense, we prove that the complete lift to the unit tangent bundle of a non-zero homothetic vector field on the base manifold is the potential vector field of a non trivial Ricci soliton structure on the unit tangent bundle endowed with an appropriately chosen pseudo-Riemannian Kaluza-Klein type metric. Furthermore, we shall give a complete classification of fiber-preserving vector fields on the unit tangent bundle which are the potential vector fields of Ricci soliton structures on the unit tangent bundle endowed with a pseudo-Riemannian Kaluza-Klein type metric.

Finally, it is worth mentioning that the existence of natural Ricci soliton structures either on tangent bundles (Theorems \ref{Solitons-comb1} and \ref{hom-flat}) or on unit tangent bundles (Theorem \ref{Ccomplet}) requires the existence of non-zero homothetic vector fields on the base manifold, which could presuppose some topological restrictions to the base manifold.

Hereafter, we will use the Einstein' summation convention.


\section{Preliminaries}



\subsection*{Basic formulas on tangent bundles}


Let $(M,g)$ be an $n$-dimensional Riemannian manifold and $\nabla$ the Levi-Civita
connection of $g$. We shall denote by $M_x$ the tangent space of $M$ at a point $x \in M$. The tangent space of $TM$ at any point $(x,u)\in TM$ splits into the horizontal and vertical subspaces with respect to $\nabla$:
$$(TM)_{(x,u)}=H_{(x,u)}\oplus V_{(x,u)}.$$ \par

For $(x,u)\in TM$ and  $X\in M_x$, there
exists a unique vector $X^h  \in H_{(x,u)}$ such that $p_* X^h =X$, where $p:TM \rightarrow M$ is the natural projection. We call $X^h$ the \emph{horizontal lift} of $X$ to the point $(x,u)\in TM$. The \emph{vertical lift} of a vector $X\in M_x$ to $(x,u)\in TM$ is the vector $X^v  \in V_{(x,u)}$ such that $X^v (df) =Xf$, for all functions $f$ on $M$. Here we consider $1$-forms $df$ on $M$ as functions on $TM$ (i.e., $(df)(x,u)=uf$).

Observe that the map $X \to X^h$ is an isomorphism between the vector spaces $M_x$ and $H_{(x,u)}$. Similarly, the map $X \to X^v$ is an isomorphism between the vector spaces $M_x$ and $V_{(x,u)}$. Obviously, each tangent vector $\tilde Z \in (TM)_{(x,u)}$ can be written in the form $\tilde Z =X^h + Y^v$, where $X,Y \in M_x$ are uniquely determined vectors.\par

Horizontal and vertical lifts of vector fields on $M$ are defined in a corresponding way. Each system of local coordinates $\{(U; x^i, i=1,..., n)\}$ in $M$ induces on $TM$ a system of local coordinates $\{(p^{-1}(U); x^i, u^i, i=1,...,n)\}$. Let $X=\sum _i X^i \left(\frac {\partial}{\partial x^i}\right)_x$ be the local expression in $\{(U; x^i, i=1,..., n)\}$ of a vector $X$ in $M_x$, $x \in M$. Then, the horizontal lift $X^h$ and the vertical lift $X^v$ of $X$ to $(x,u) \in TM$ are given, with respect to the induced coordinates, by:
\begin{equation}\label{lift2}
\begin{array}{l}
 X^h = \sum X^i \left(\frac {\partial}{\partial x^i}\right)_{(x,u)}
\, -\,\sum \Gamma_{jk}^iu^jX^k \left(\frac {\partial}{\partial
u^i}\right)_{(x,u)}
\end{array}
\end{equation}
and
\begin{equation}
\begin{array}{l}
X^v  = \sum
X^i \left(\frac {\partial}{\partial u^i}\right)_{(x,u)},
\label{lift3}
\end{array}
\end{equation}
where $(\Gamma_{jk}^i)$ denote the Christoffel's symbols of $g$. \par

Let $K:TTM \rightarrow TM$ be the connection map corresponding to the Levi-Civita connection $\nabla$ of $(M,g)$. Note that $K$ is characterized by $K(X^h)=0$ and $K(X^v)=X$, for all $X \in TM$.\par

The {\em canonical vertical vector field $\mathcal{U}$} on $TM$ is  defined, in terms of local coordinates, by $\mathcal{U}=\sum_i u^i \partial/\partial u^i$. Here $\mathcal{U}$
does not depend on the choice of local coordinates and is defined globally on $TM$. For a vector $u=\sum_i u^i (\partial/\partial x^i)_x \in M_x$, we see that $u^v_{(x,u)}= \sum_i u^i (\partial/\partial x^i)^v_{(x,u)}= \mathcal{U}_{(x,u)}$ and $u^h_{(x,u)}= \sum_i u^i (\partial/\partial x^i)^h_{(x,u)}$ (which is also known as the {\em geodesic vector field}).\par

There are three other interesting vector fields on the tangent bundle obtained by lifting operations of geometric objects on $M$: For any vector field $X$ and a $(1,1)$-tensor field $P$ on $M$, we define the vector fields $X^c$, $\iota P$ and $^*P$ on $TM$, by
\begin{eqnarray}
X^c_{(x,u)} &=& X^h_{(x,u)}+ (\nabla_u X)^v_{(x,u)}, \\
(\iota P)_{(x,u)} &=& [P(u)]^v_{(x,u)}, \\
(^*P)_{(x,u)} &=& [P(u)]^h_{(x,u)},
\end{eqnarray}
for all $(x,u) \in TM$. $X^c$ is called the \emph{complete lift} of $X$. It is easy to see that $X^c=X^h+ \iota(\nabla X)$. \par

The Riemannian curvature $R$ of $g$ is defined by
\begin{equation}\label{r-cur}
R(X,Y)=[\nabla_X,\nabla_Y]\,-\,\nabla_{[X,Y]}.
\end{equation}\par
%


\subsection*{$g$-natural metrics on tangent bundles}


There are three distinguished constructions of metrics on the tangent bundle $TM$ \\
(a) the Sasaki metric $g^s$ is deﬁned by
\begin{center}
	$\begin{array}{lll}
	g_{(x,u)}^{s}(X^h,Y^h)&=	&g(X,Y)
	\\
	g_{(x,u)}^{s}(X^v,Y^h)& =&0\\
	\end{array}$
	$\begin{array}{lll}
	g_{(x,u)}^{s}(X^h,Y^v)&=	&0
	\\
	g_{(x,u)}^{s}(X^v,Y^v)& =&g(X,Y)\\
	\end{array}$
\end{center}
for all $X, Y\in M_x.$\\ (b) The horizontal lift $g^h$ of $g$ is a pseudo-Riemannian metric on $TM,$ given by
\begin{center}
	$\begin{array}{lll}
	g_{(x,u)}^{h}(X^h,Y^h)&=	&0
	\\
	g_{(x,u)}^{h}(X^v,Y^h)& =&g(X,Y)\\
	\end{array}$
	$\begin{array}{lll}
	g_{(x,u)}^{h}(X^h,Y^v)&=	&g(X,Y)
	\\
	g_{(x,u)}^{h}(X^v,Y^v)& =&0\\
	\end{array}$
\end{center}
for all $X,Y \in M_x.$ \\ (c) The vertical lift $g^v$ of $g$ is a degenerate metric on $TM,$ given by
\begin{center}
	$\begin{array}{lll}
	g_{(x,u)}^{v}(X^h,Y^h)&=	&g(X,Y)
	\\
	g_{(x,u)}^{v}(X^v,Y^h)& =&0\\
	\end{array}$
	$\begin{array}{lll}
	g_{(x,u)}^{v}(X^h,Y^v)&=	&0
	\\
	g_{(x,u)}^{v}(X^v,Y^v)& =&0\\
	\end{array}$
\end{center}
for all $X, Y \in M_x.$\\
Starting from a Riemannian manifold $(M,g)$, a natural construction leads to introduce a wide class of metrics, called {\em $g$-natural},  on the tangent bundle $TM$ (\cite{Kol-Mic-Slo},\cite{Kow-Sek1}). Such metrics are the image of $g$ under first order natural operators $D:S_+^2T^* \rightsquigarrow (S^2T^*)T$, which transform Riemannian metrics on manifolds into metrics on their tangent bundles, where
$S_+^2T^*$ and $S^2T^*$ denote the bundle functors of all
Riemannian metrics and all symmetric $(0,2)$-tensors over $n$-manifolds respectively.

Given an arbitrary $g$-natural metric $G$ on the tangent bundle $TM$ of a Riemannian manifold $(M,g)$, there exist six smooth functions $\alpha_i$, $\beta_i:\mathbb{R}^+ \rightarrow \mathbb{R}$, $i=1,2,3$, such that (see \cite{Abb-Sar4})
\arraycolsep1.5pt
\begin{equation}\label{exp-$g$-nat}
 \left\lbrace
\begin{array}{rcl}
G_{(x,u)}(X^h,Y^h)& = & (\alpha_1+ \alpha_3)(r^2) g_x(X,Y)+ (\beta_1+ \beta_3)(r^2)g_x(X,u)g_x(Y,u),\\
G_{(x,u)}(X^h,Y^v)& = & G_{(x,u)}(X^v,Y^h)\\
                  & = & \alpha_2 (r^2) g_x(X,Y)+  \beta_2 (r^2) g_x(X,u)g_x(Y,u), \\
G_{(x,u)}(X^v,Y^v)& = & \alpha_1 (r^2) g_x(X,Y) + \beta_1 (r^2) g_x(X,u)g_x(Y,u),
\end{array}
\arraycolsep5pt \right.
\end{equation}
for every $u$, $X$, $Y\in M_x$, where $r^2 =g_x(u,u)$.   Put
\begin{equation*}
\begin{split}
\phi_i(t)= & \alpha_i(t) +t \beta_i(t), \quad \alpha(t) = \alpha_1(t) (\alpha_1+\alpha_3)(t) - \alpha_2 ^2(t), \\
    & \phi(t) = \phi_1(t) (\phi_1 +\phi_3)(t) -\phi_2^2(t),
\end{split}
\end{equation*}
for all $t \in \mathbb{R}^+$. 
%
%
{It it easily seen that $G$ is}
\begin{itemize}
\item non-degenerate if and only if
$$
\alpha(t) \neq 0, \qquad  \phi(t) \neq 0 \qquad \text{for all} \;\, t \in \mathbb{R}^+ ;
$$
\item
Riemannian if and only if
$$
\alpha_1(t) > 0,  \qquad \phi_1(t) > 0, \qquad \alpha(t) > 0, \qquad  \phi(t)>0 \qquad \text{for all} \;\, t \in \mathbb{R}^+ .
$$
\end{itemize}

The wide class of $g$-natural metrics includes several well known metrics (Riemannian and not) on $TM$. In particular:
\begin{itemize}
\item the {\em Sasaki metric} $g_S$ is obtained for $\alpha _1 =1$ and $\alpha _2 = \alpha _3 = \beta _1 =\beta _2 = \beta _3 =0$.
\item {\em Kaluza--Klein metrics}, as commonly defined on principal bundles {(see for example \cite{Wood})}, are obtained for
$\alpha _2 = \beta _2 = \beta _1 +\beta _3 = 0$.
\item {\em Metrics of Kaluza--Klein type} are defined by the geometric condition of orthogonality between horizontal and vertical distributions \cite{CP,CP3}. Thus, a  $g$-natural metric $G$ is of Kaluza-Klein type if $\alpha _2=\beta _2 =0$.
\end{itemize}


\subsection*{$g$-natural metrics on unit tangent bundles}


The  \emph{ unit tangent bundle} over a Riemannian manifold $(M,g)$ is the hypersurface of $TM$, given by
$$T_1 M= \{(x,u) \in TM \, | \, g_x(u,u)=1 \}.$$
We will denote by $p_1 : T_1 M \to M$ the bundle projection. The tangent space of $T_1 M$ at a point $(x,u) \in T_1M$ is given by
\begin{equation}\label{tang-space}
(T_1 M)_{(x,u)}=\{X^h +Y^v /X \in M_x,  Y \in \{u\}^\perp \subset M_x\}.
\end{equation}\par

By definition, {\em $g$-natural metrics on the unit tangent bundle} are the metrics induced on the hypersurface $T_1 M$ by corresponding $g$-natural metrics on $TM$. As proved in \cite{Abb-Kow2} for the Riemannian case, and extended to pseudo-Riemannian settings in \cite{CMM}, they are completely determined by the values of the four
real constants
$$a:=\alpha_1 (1), \quad b:=\alpha_2 (1), \quad c:=\alpha_3 (1), \quad d:=(\beta_1+\beta_3) (1)
$$
and  admit the following explicit description.

\begin{Th}[\cite{Abb-Kow2}] 	\label{nat-met-unit}
	Let $(M,g)$ be a Riemannian manifold. For every pseudo-Riemannian metric $\tilde{G}$ on $T_1 M$ induced from a $g$-natural metric $G$ on $TM$, there exist four constants $a$, $b$, $c$ and $d$, satisfying the inequalities
	$$a \neq 0, \quad \alpha:=a(a+c)-b^2 \neq 0, \quad \varphi:=a+c+d \neq 0$$
	(in particular, they are Riemannian if $a, \alpha, \varphi>0$), such that
	\begin{equation}\label{exp-metr-0}
	\left\lbrace \arraycolsep1.5pt
	\begin{array}{lcl}
	\tilde{G}_{(x,u)}(X_1^h,X_2^h)  &=& (a+c) g_x(X_1,X_2) +d g_x(X_1,u)g_x(X_2,u), \\[2pt]
	\tilde{G}_{(x,u)}(X_1^h,Y_1^v)  &=& b g_x(X_1,Y_1),\\[2pt]
	\tilde{G}_{(x,u)}(Y_1^v,Y_2^v)  &=& a g_x(Y_1,Y_2) ,
	\end{array}
	\right.
	\end{equation}
	for all $(x,u) \in T_1 M$, $X_1$, $X_2 \in M_x$ and $Y_1$, $Y_2 \in
	\{u\}^\perp \subset M_x$.
\end{Th}

\noindent
In particular, the {\em Sasaki metric} on $T_1 M$ corresponds to the case where $a=1$ and $b=c=d=0$; {\em Kaluza-Klein metrics} are obtained when $b=d=0$; {\em metrics of Kaluza-Klein type} are given by the case $b=0$. In the remaining part of the paper,  we shall restrict ourselves to the latter type of $g$-natural metrics on $T_1 M$.

Since $G$  must be nondegenerate, throughout the paper we implicitly assume that inequalities $a \neq 0$, $a+c \neq 0$ and $\varphi:=a+c+d \neq 0$ hold (they are exactly the special case of inequalities from Theorem \ref{nat-met-unit} for $b=0$).
\par

For any vector field $Z$ on $TM$, we define its tangential component $t\{Z\}$ at points of $T_1M$, by $t\{Z\}:=Z|_{T_1M} -G(Z|_{T_1M},N)N$, obtaining a vector field on $T_1M$. When $Z$ is the horizontal, the vertical or the complete lift of a vector field $X$ on $M$ or the vector field $\iota P$, for a $(1,1)$-tensor field $P$ on $M$, then we obtain, respectively, the following special vector fields on $T_1M$:
\begin{itemize}
	\item [a-] the \emph{horizontal lift} to $T_1M$ of $X$, denoted also $X^h$, and given by $X^h:=X^h|_{T_1M}$, since $X^h$ is orthogonal to $N$ everywhere on $T_1M$;
	\item [b-] the \emph{tangential lift} $X^{t} $ with respect to $G$ of $X$, defined as the tangential component vector field of the vertical lift $X^v$ of $X$, that is,
	\begin{equation}\label{tan}
	X^{t}_{(x,u)} =[X-g_x(u,X)u]^v_{(x,u)}+\frac{b}{\varphi}g_x(u,X)u^h_{(x,u)},
	\end{equation}
	for all $(x,u) \in T_1M$. If $X_x \in M_x$ is orthogonal to $u,$ then $X^t_{(x,u)}= X^v_{(x,u)}$;
	\item [c-] the \emph{complete lift} to $T_1M$ of $X$, denoted by $X^{\bar{c}}$, is given by
	\begin{equation}\label{comp}
	X^{\bar{c}}_{(x,u)}=X^h_{(x,u)}+ (\nabla_u X)^t_{(x,u)},
	\end{equation}
	for all $(x,u) \in T_1M$. It is worth mentioning that the complete lift $X^c$ on $TM$, restricted to $T_1M$, coincides with $X^{\bar{c}}$ if and only if $X$ is a Killing vector field on $M$;
	\item [d-] the vector field $\tilde{\iota} P$ on $T_1M$, given by
	\begin{equation}\label{iot}
	(\tilde{\iota} P)_{(x,u)}= [P(u)]^t_{(x,u)},
	\end{equation}
	for all $(x,u) \in T_1M$.
\end{itemize}

We will make use of the following relations, which are easy to check:
\begin{eqnarray}
&& X^h(f \circ p_1) = X(f) \circ p_1, \label{Pr1}\\
&& X^t(f \circ p_1)=0,   \label{Pr2}\\
&&  X^h_{(x,u)}(g(Y,.)) = g(\nabla_{X_x}Y,u)  \label{Pr3}\\
&& X^t_{(x,u)}(g(Y,.))=g(X_x,Y_x)-g(X_x,u)g(Y_x,u), \label{Pr4}
\end{eqnarray}
for all $X,Y \in \mathfrak{X}(M)$, $f\in C^\infty(M)$ and $(x,u) \in T_1M$.

\begin{Conv}\label{conv}
	Because $u^{t}_{(x,u)}=0$, for $(x,u) \in T_1 M$, which means that the tangent space $(T_1 M)_{(x,u)}$ coincides with the set
	\begin{equation}\label{tang-space-G}
	\{X^h +Y^{t} /X \in M_x, Y \in \{u\}^\perp \subset M_x\},
	\end{equation}
	the operation of tangential lift from $M_x$ to a point $(x,u) \in T_1 M$ will be always applied only to vectors of $M_x$ which are orthogonal to $u$.
\end{Conv}

Assuming that $(M,g)$ is of constant curvature $\kappa$, for a $g$-natural metric $\widetilde{G}$ of Kaluza-Klein type, the following formulas are deduced from the expressions of the Levi-Civita connection of $(T_1 M,\widetilde{G})$ given in \cite{Abb-Kow1}:

\begin{Pro}\label{cor-LC}
	If $(M,g)$ is a Riemannian manifold of constant sectional curvature $\kappa$ and $\widetilde{G}$ is a pseudo-Riemannian $g$-natural metric of Kaluza-Klein type on $T_1M$, then the Levi-Civita connection  $\widetilde{\nabla}$ of $(T_1 M,\widetilde{G})$ is given by: \arraycolsep2pt
	\begin{equation*}
	\begin{array}{ll}
	(\widetilde{\nabla}_{X^h}Y^h)_{(x,u)} =& \{\nabla_XY\}^h_{(x,u)} +\frac{1}{2a}\left\{ -(a\kappa+d)g(Y_x,u)X_x \frac{}{}\right.\\[4pt]
	&  \left.\frac{}{}+(a\kappa-d)g(X_x,u)Y_x \right\}^{t}_{(x,u)},\\[6pt]
	(\widetilde{\nabla}_{X^h}Y^{t})_{(x,u)} =& \left\{\frac{d-a\kappa}{2(a+c)} g(X_x,u)Y_x\right.+\frac{1}{2\varphi}\left[(a\kappa+d) g(X_x,Y_x)\frac{}{}\right.\\[4pt]
	&  \left.\left.+\frac{d(a\kappa+d-2\varphi)}{a+c}g(X_x,u)g(Y_x,u)\right]u\right\}^h_{(x,u)} +\{\nabla_XY\}^{t}_{(x,u)}, \\[6pt]
	(\widetilde{\nabla}_{X^{t}}Y^{h})_{(x,u)} =& \left\{\frac{d-a\kappa}{2(a+c)} g(Y_x,u)X_x\right.+\frac{1}{2\varphi}\left[(a\kappa+d) g(X_x,Y_x)\frac{}{}\right.\\[4pt]
	&  \left.\left.+\frac{d(a\kappa+d-2\varphi)}{a+c}g(X_x,u)g(Y_x,u)\right]u\right\}^h_{(x,u)}, \\[6pt]
	(\widetilde{\nabla}_{X^{t}}Y^{t})_{(x,u)}  =&  -\{g(Y_x,u)X_x\}^{t}_{(x,u)},
	\end{array}
	\end{equation*}
	for all $x \in M$, $(x,u) \in T_1 M$ and all arbitrary vector fields $X$, $Y \in \mathfrak{X}(M)$.
\end{Pro}

Using Proposition  \ref{cor-LC}, we can prove the following:

\begin{Lem}\label{Lie-tang}
	For all $\xi \in \mathfrak{X}(M)$, $(x,u) \in TM$ and $X, Y \in T_xM$, we have
	\begin{equation*}
	\begin{split}
	\left(\mathcal{L}_{\xi^{t}}\tilde{G}\right)_{(x,u)}(X^h,Y^h)=& d\{g(\xi_x,X)g(Y,u)+ g(\xi_x,Y)g(X,u)\\
	& - 2g(\xi_x,u)g(X,u)g(Y,u)\},  \\
	&  \\
	\left(\mathcal{L}_{\xi^{t}}\tilde{G}\right)_{(x,u)}(X^h,Y^t)= & a\{g(\nabla_X \xi,Y) -g(Y,u)g(\nabla_X \xi,u)\},\\
	& \\
	\left(\mathcal{L}_{\xi^{t}}\tilde{G}\right)_{(x,u)}(X^t,Y^t)=
	& -2ag(\xi_x,u)\{g(X,Y) -g(X,u)g(Y,u)\}.
	\end{split}
	\end{equation*}
\end{Lem}
\begin{Lem}\label{Lie-hor}
	For all $\xi \in \mathfrak{X}(M)$, $(x,u) \in TM$ and $X, Y \in T_xM$, we have
	\begin{equation*}
	\begin{split}
	\left(\mathcal{L}_{\xi^{h}}\tilde{G}\right)_{(x,u)}(X^h,Y^h)= & (a+c)\{g(\nabla_X \xi,Y) +g(\nabla_Y \xi,X)\}, \\
	& +d\{g(\nabla_X \xi,u)g(Y,u) +g(\nabla_Y \xi,u)g(X,u)\},  \\
	&  \\
	\left(\mathcal{L}_{\xi^{h}}\tilde{G}\right)_{(x,u)}(X^h,Y^t)= & a\kappa\{g(\xi_x,Y)g(X,u)-g(\xi_x,u)g(X,Y)\}, \\
	& \\
	\left(\mathcal{L}_{\xi^{h}}\tilde{G}\right)_{(x,u)}(X^t,Y^t)= & 0.
	\end{split}
	\end{equation*}
\end{Lem}

Next, the Ricci tensor $\widetilde{\textup{Ric}}$ of $(T_1 M,\widetilde{G})$ is given by the following.

\begin{Pro}\label{cor-ric}{\em\cite{Abb-Kow3}}
	If $(M,g)$ is a space of constant sectional curvature $\kappa$ and $\widetilde{G}$ is a pseudo-Riemannian $g$-natural metric of Kaluza-Klein type on $T_1M$, then
	the Ricci tensor $\widetilde{\textup{Ric}}$ of $(T_1 M,\widetilde{G})$
	is given by: \arraycolsep2pt
	\begin{eqnarray}
	\qquad \widetilde{\textup{Ric}}_{(x,u)} (X^h,Y^h) &= & \frac{1}{2a\varphi}
	[-a^2\kappa^2 +2(n-1)a\varphi \kappa +d(d-2\varphi)]\,
	g(X,Y) \label{ric1k}\\
	&  & +\frac{1}{2a(a+c)\varphi}[-a^2((n-2)\varphi +d)\kappa^2
	+d(2n(a+c)\varphi  \nonumber\\
	&  & +(n-1)d\varphi -d(a+c))]\, g(X,u)g(Y,u),\nonumber \\[4pt]
	\qquad \widetilde{\textup{Ric}}_{(x,u)} (X^h,Y^t)& = &0 ,\label{ric2k}\\[4pt]
	\qquad  \widetilde{\textup{Ric}}_{(x,u)} (X^{t},Y^{t}) &= &\frac{1}{2(a+c)\varphi}
	[a^2\kappa^2 +2(n-2)(a+c)\varphi -d^2]\, g(X,Y),\label{ric3k}
	\end{eqnarray}
	for all $x \in M$, $(x,u) \in T_1 M$ and all arbitrary vectors $X$, $Y \in M_x$ satisfying Convention \ref{conv}.
\end{Pro}


\section{Natural Ricci Solitons on tangent bundles}


\begin{Th}
Let $(M,g)$ be a Riemannian manifold of dimension $n \geq 3$  and  $G$ be a pseudo-Riemannian $g$-natural metric on $TM$ whose functions $\alpha_i, \beta_i$ $i=1,2,3$ satisfy $\alpha_1(0)(\alpha_1+\alpha_3)(0)-2\alpha_{2}^{2}(0)\neq 0$.  If $(TM,G,Z,\bar{\lambda})$ is a Ricci Soliton, then  $(M,g,Z_0,\lambda)$ is a Ricci soliton with
\begin{itemize}
	\item $Z_0(x)=\frac{\alpha(0)}{[\alpha_1(\alpha_1+\alpha_3)-2\alpha_{2}^{2}](0)} \{(\alpha_1+\alpha_3)(0)d\pi(Z(x,0))+\alpha_2(0)K(Z(x,0))\}$;
	\item $\lambda=\frac{(\alpha_1+\alpha_3)[\bar{\lambda}\alpha+\beta_1+\beta_3+n(\alpha_1+\alpha_3)']}{\alpha_1(\alpha_1+\alpha_3)-2\alpha_{2}^{2}}(0)$,
\end{itemize}
and $K$ is the connection map.
\end{Th}
\begin{Proof}
	 Fix $x\in M$ and let $(E_1,...,E_n)$ be an orthonormal basis of $(M_x,g_x)$. We will distinguish two cases:

\textbf{Case 1:} $(\alpha_1+\alpha_3)(0) \neq 0$. If we put
\begin{equation*}
  \begin{split}
    F_i= & \frac{1}{\sqrt{|(\alpha_1+\alpha_3)(0)|}}E_{i}^{h}(x,0), \\
    F_{i+n}=  & \frac{1}{\sqrt{|\alpha(0)(\alpha_1+\alpha_3)(0)|}}\left[\alpha_2(0)E_{i}^{h}(x,0)-(\alpha_1+\alpha_3)(0)E_{i}^{v}(x,0)\right],
  \end{split}
\end{equation*}
$i=1,...,n$, then we obtain an orthonormal basis $\{F_1,...,F_{2n}\}$  of the tangent space $((TM)_{(x,0)},G_{(x,0)})$. If we put $\varepsilon_I:=G(F_I,F_I)$, $I=1,...,2n$, then the Ricci  curvature of $(TM,G)$ at $(x,0)$ is then given by
$$\overline{\textup{Ric}}_{(x,0)}(V,W)=\sum_{I=1}^{2n}\varepsilon_IG_{(x,0)}(\bar{R}(V,F_I)F_I,W),$$
for every $V,W \in (TM)_{(x,0)}$.

Taking $V=X^h$ and $W=Y^h$, for $X,Y \in M_x$, then simple calculation gives
	\begin{equation}\label{Ric1}
	\begin{array}{lll}
	\overline{\textup{Ric}}_{(x,0)}(X^h,Y^h) &= &\frac{1}{\alpha(0)}\sum_{i=1}^{n}\left\{\alpha_1(0)G_{(x,0)}(\bar{R}(X^h,E_{i}^{h})E_{i}^{h},Y^h)\right.\\
&&\quad+(\alpha_1+\alpha_3)(0)G_{(x,0)}(\bar{R}(X^h,E_{i}^{v})E_{i}^{v},Y^h)
	\\
	& &\quad-\alpha_2(0)[G_{(x,0)}(\bar{R}(X^h,E_{i}^{h})E_{i}^{v},Y^h)\\
&&\left.\qquad\qquad+G_{(x,0)}(\bar{R}(X^h,E_{i}^{v})E_{i}^{h},Y^h)]\right\}. \\
	\end{array}
		\end{equation}
	But, restricting the Riemannian curvature $\bar{R}$ of $(TM,G)$ to the zero section (cf. \cite[Proposition 3.1]{Abb-Sar4}), we obtain
	
	$\begin{array}{lll}
	G_{(x,0)}(\bar{R}(X^h,E_{i}^{h})E_{i}^{h},Y^h) &= &(\alpha_1+\alpha_3)(0)g(R(X,E_i)E_i,Y),
	\\
	G_{(x,0)}(\bar{R}(X^h,E_{i}^{h})E_{i}^{v},Y^h)&=& \alpha_2(0)g(R(X,E_i)E_i,Y), \\
	G_{(x,0)}(\bar{R}(X^h,E_{i}^{v})E_{i}^{h},Y^h)&=&  \alpha_2(0)g(R(X,E_i)E_i,Y),\\
	G_{(x,0)}(\bar{R}(X^h,E_{i}^{v})E_{i}^{v},Y^h)&=&-(\alpha_1+\alpha_3)'(0)g(X,Y)\\
&&-(\beta_1+\beta_3)(0)g(X,E_i)g(Y,E_i). \\
	\end{array}$

Replacing from the last identities into \eqref{Ric1}, we deduce that
	
	$\begin{array}{lll}
	\overline{\textup{Ric}}_{(x,0)}(X^h,Y^h) &= &\frac{1}{\alpha(0)}\sum_{i=1}^{n}\{\left[\alpha_1(\alpha_1+\alpha_3)-2\alpha_{2}^{2}\right](0)g(R(X,E_i)E_{i},Y)
	\\
	& &-(\alpha_1+\alpha_3(0)[(\alpha_1+\alpha_3)'(0)g(X,Y)\\
&&+(\beta_1+\beta_3)(0)g(X,E_i)g(Y,E_i)]\}.
	\end{array}$ \\
	Since $(E_1,...,E_n)$ is an orthonormal basis of $ (M_x,g_x) $, we deduce that
	\begin{equation}\label{Ricci1}
	\begin{split}
   \overline{\textup{Ric}}_{(x,0)}(X^h,Y^h)= & \frac{1}{\alpha(0)}[((\alpha_1+\alpha_3)-2\alpha_{2}^{2})(0)\textup{Ric}(X,Y)\\
     & -(\alpha_1+\alpha_3)(0)((\beta_1+\beta_3)(0)+n(\alpha_1+\alpha_3)'(0))g(X,Y)].
 \end{split}
	\end{equation}
	Now, if $ (TM, G, Z, \bar {\lambda}) $ is a Ricci Soliton, then we have in particular:
	\begin{equation}\label{Riccie}
	\bar{R}ic_{(x,0)}(X^h,Y^h)+\frac{1}{2}(\mathcal{L}_ZG)_{(x,0)}(X^h,Y^h)=\bar{\lambda} G_{(x,0)}(X^h,Y^h).
	\end{equation}
	Expressing $Z$, in an induced local system $(p^{-1}(U);x^i,u^i,i=1,...,n)$ of $TM$, as $Z=\sum_{l=1}^{n} \left[ A^l \frac{\partial}{\partial x^l} +B^l \frac{\partial}{\partial u^l}\right]$, and restricting the Levi-Civita connection $\bar{\nabla}$ of $(TM,G)$ to the zero section (cf. \cite[Proposition 1.5]{Abb-Sar4}), we obtain
	\begin{equation}\label{lie11}
	\begin{array}{ll}
	&(\mathcal{L}_ZG)_{(x,0)}(X^h,Y^h)=G_{(x,0)}(X^h,\bar{\nabla}_{Y^h}Z)+G_{(x,0)}(Y^h,\bar{\nabla}_{X^h}Z)
	\\
	 &\qquad=(\alpha_1+\alpha_3)(0)[g(X,Y(A^l(x,0))\frac{\partial}{\partial x^l}\arrowvert_x)+g(Y,X(A^l(x,0))\frac{\partial}{\partial x^l}\arrowvert_x)\\
	& \qquad\quad+A^l(x,0)(g(X,\nabla_Y\frac{\partial}{\partial x^l})+g(Y,\nabla_X\frac{\partial}{\partial x^l}))] \\
	&\qquad\quad+\alpha_2(0)[g(X,Y(B^l(x,0))\frac{\partial}{\partial x^l}\arrowvert_x)+g(Y,X(B^l(x,0))\frac{\partial}{\partial x^l}\arrowvert_x) \\
	& \qquad\quad+B^l(x,0)(g(X,\nabla_Y\frac{\partial}{\partial x^l})+g(Y,\nabla_X\frac{\partial}{\partial x^l}))].
	\end{array}
		\end{equation}
Setting, for every $ x  \in U$,
$$W(x)=[(\alpha_1+\alpha_3)(0)A^l(x,0)+\alpha_2(0)B^l(x,0)]\left.\frac{\partial}{\partial x^l}\right|_{x},$$
then $W$ is a vector field on $U,$ and \eqref{lie11} becomes
\begin{equation}\label{Lied}
	(\mathcal{L}_ZG)_{(x,0)}(X^h,Y^h)=(\mathcal{L}_Wg)_x(X,Y).	
\end{equation}
Substituting from \eqref{Ricci1} and \eqref{Lied} into \eqref{Riccie}, we deduce, under $(\alpha_1(\alpha_1+\alpha_3)-2\alpha_{2}^{2})(0)\neq 0 $, that $$\textup{Ric}(X,Y)+\frac{1}{2}\mathcal{L}_{Z_0}g(X,Y)=\lambda g(X,Y),$$
for every $X,Y \in M_x$, where $Z_0=\frac{\alpha(0)}{(\alpha_1(\alpha_1+\alpha_3)-2\alpha_{2}^{2})(0)}W$ and $$\lambda=\frac{(\alpha_1+\alpha_3)[\bar{\lambda}\alpha+\beta_1+\beta_3+n(\alpha_1+\alpha_3)']}{\alpha_1(\alpha_1+\alpha_3)-2\alpha_{2}^{2}}(0).$$

\textbf{Case 2:} $(\alpha_1 +\alpha_3)(0)=0$. In this case $\alpha_2(0) \neq 0$. We put
\begin{equation*}
  F_i= \frac{1}{\sqrt{2|\alpha_2(0)|}}[E_i^h +E_i^v], \qquad \textup{and} \qquad F_{n+i}= \frac{1}{\sqrt{2|\alpha_2(0)|}}[E_i^h -E_i^v],
\end{equation*}
$i=1,...,n$. As in the first case, $\{F_I; I=1,...,2n\}$ is an orthonormal basis of $((TM)_{(x,0)},G_{(x,0)})$, and in a similar way we obtain $$\overline{\textup{Ric}}_{(x,0)}(X^h,Y^h)= 2\textup{Ric}(X,Y),$$
and then $\textup{Ric}(X,Y)+\frac{1}{2}\mathcal{L}_{Z_0^\prime}g(X,Y)=0$, for any $x \in M$ and $X$, $Y \in M_x$, where $Z_0^\prime(x):=-\alpha_2(0)B^l(x,0)\left.\frac{\partial}{\partial x^l}\right|_{x}$, i.e. $(TM,G,Z_0^\prime,0)$ is a Ricci Soliton. Remark that, for $(\alpha_1+\alpha_3)(0)=0$, we obtain $\bar{\lambda}=0$ and $Z_0=Z_0^\prime$. This completes our proof.
\end{Proof}

\begin{Cor}
	Let $ (M, g) $ be a Riemannian manifold of dimension $ n \geq 3 $ and $ G $ be a pseudo-Riemannian $ g $-natural metric on $ TM $ whose functions  $\alpha_i, \beta_i$ $i=1,2,3$ satisfy $\alpha_1(0)(\alpha_1+\alpha_3)(0)-2\alpha_{2}^{2}(0)\neq 0.$ If $(TM,G,\bar{f},\bar{\lambda})$ is a gradient type Ricci soliton, then  $(M,g,f,\lambda)$ is a gradient type Ricci soliton with
		\begin{itemize}
			\item $f(x)=\frac{\alpha(0)}{\alpha_1(0)(\alpha_1+\alpha_3)(0)-2\alpha_{2}^{2}(0)}\bar{f}(x,0) $, for all $x \in M$;
			\item $\lambda=\frac{(\alpha_1+\alpha_3)[\bar{\lambda}\alpha+\beta_1+\beta_3+n(\alpha_1+\alpha_3)']}{\alpha_1(\alpha_1+\alpha_3)-2\alpha_{2}^{2}}(0)$.
		\end{itemize}
\end{Cor}


\section{$g$-natural metrics of the form $ag^s+bg^h+cg^v$ }


\begin{Th}\label{Solitons-comb1}
	Let $ (M, g) $ be a Riemannian manifold of dimension $n \geq 2$,  $G=ag^s+bg^h+cg^v$ such that $a \neq 0$, $a+c \neq 0$ and $a(a+c) -b^2 \neq 0$, $Z$ be a vector field on $TM$  and $\bar{\lambda}\in \mathbb{R}$. Then $(TM,G,Z,\bar{\lambda})$ is a Ricci Soliton if and only if
	\begin{enumerate}
		\item $(M,g)$ flat;
		\item $Z$ is a homothetic vector field on $(TM,G)$, with homothety factor $\bar{\lambda}$.
	\end{enumerate}
\end{Th}

\begin{Proof}
	Suppose that $(TM,G,Z,\bar{\lambda})$ is a Ricci Soliton. For $x$  in $M$, let $(x^1,...,x^n)$ be a normal coordinate system at  $x$ and $(\pi^{-1}(U),x^1,...,x^n,u^1,...,u^n)$ its induced coordinate system on $ TM$. With respect to these induced coordinates, $Z$ can be expressed as $Z=\sum_{l=1}^{n} \left[ A^l \frac{\partial}{\partial x^l} +B^l \frac{\partial}{\partial u^l}\right]$. We have in particular
	\begin{equation}\label{Solitons-comb11}
	\overline{\textup{Ric}}(\frac{\partial}{\partial u^i},\frac{\partial}{\partial u^j})+\frac{1}{2}\left(\mathcal{L}_ZG\right)(\frac{\partial}{\partial u^i},\frac{\partial}{\partial u^j})=\bar{\lambda}G(\frac{\partial}{\partial u^i},\frac{\partial}{\partial u^j}),
	\end{equation}
   for $i,j=1,...,n$. Using the expression of the Levi-Civita connection (cf. \cite[p.32]{Abb-Sar5}) and the fact that $(x^1,...,x^n)$ is normal at $x$, we have
\begin{equation}
	\begin{array}{rl}
	\left(\mathcal{L}_ZG\right)_{(x,u)}&(\frac{\partial}{\partial u^i},\frac{\partial}{\partial u^j})  =  G_{(x,u)}(\frac{\partial}{\partial u^i},\bar{\nabla}_{\frac{\partial}{\partial u^j}}Z) +G_{(x,u)}(\frac{\partial}{\partial u^j},\bar{\nabla}_{\frac{\partial}{\partial u^i}}Z)\\
 = & \frac{\partial}{\partial u^j}|_{(x,u)}(bA^i+aB^i)+\sum_{l=1}^n A^lG_{(x,u)}(\frac{\partial}{\partial u^i},\bar{\nabla}_{\frac{\partial}{\partial u^j}}(\frac{\partial}{\partial x^l})^h)\\
  & +\frac{\partial}{\partial u^i}|_{(x,u)}(bA^j+aB^j)+\sum_{l=1}^n A^lG_{(x,u)}(\frac{\partial}{\partial u^j},\bar{\nabla}_{\frac{\partial}{\partial u^i}}(\frac{\partial}{\partial x^l})^h)
	\\
	= & \frac{\partial}{\partial u^j}|_{(x,u)}(bA^i+aB^i) +\frac{\partial}{\partial u^i}|_{(x,u)}(bA^j+aB^j), \\
	\end{array}
	\end{equation}
for all $(x,u) \in M_x$. On the other hand, using \cite[Proposition 3.1]{Abb-Sar5} and the fact that $(x^1,...,x^n)$ is normal at $x$, we have
	\begin{equation}
	\overline{\textup{Ric}}_{(x,u)}\left(\frac{\partial}{\partial u^i}, \frac{\partial}{\partial u^j}\right)= \frac{a^4}{4\alpha^2}\sum_{l=1}^{n}g\left(R\left(u,\left.\frac{\partial}{\partial x^i}\right|_x\right)\left.\frac{\partial}{\partial x^l}\right|_x,R\left(u,\left.\frac{\partial}{\partial x^j}\right|_x\right)\left.\frac{\partial}{\partial x^l}\right|_x\right).
	\end{equation}
	So we get
	\begin{equation}\label{Di}
\begin{split}
  &\left.\frac{\partial}{\partial u^j}\right|_{(x,u)}(bA^i+aB^i)+\left.\frac{\partial}{\partial u^i}\right|_{(x,u)}(bA^j+aB^j)= \\
   = & 2\left\{\bar{\lambda}a\delta_{ij}-\frac{a^4}{4\alpha^2}\sum_{l=1}^{n}g\left(R\left(u,\left.\frac{\partial}{\partial x^i}\right|_x\right)\left.\frac{\partial}{\partial x^l}\right|_x,R\left(u,\left.\frac{\partial}{\partial x^j}\right|_x\right)\left.\frac{\partial}{\partial x^l}\right|_x\right)\right\}.
\end{split}
	 \end{equation}
	 Let us define the functions $D_i$, $i=1,...,n$, on $M_x$ by
$$D_i=\frac{2\alpha^2}{a^4}[\bar{\lambda}au^i-(bA^i+aB^i)].$$
So \eqref{Di} reduces to
	\begin{equation}\label{DD}
	\left.\frac{\partial D_i}{\partial u^j}\right|_{(x,u)}+\left.\frac{\partial D_j}{\partial  u^i}\right|_{(x,u)}=\sum_{l=1}^{n}g\left(R\left(u,\left.\frac{\partial}{\partial x^i}\right|_x\right)\left.\frac{\partial}{\partial x^l}\right|_x,R\left(u,\left.\frac{\partial}{\partial x^j}\right|_x\right)\left.\frac{\partial}{\partial x^l}\right|_x\right),
	\end{equation}
$i,j=1,...,n$.
	
	For $i=j$, we get
\begin{equation*}
  \begin{split}
    \frac{\partial D_i}{\partial u^i}|_{(x,u)}= & \sum_{l=1}^{n}\frac{1}{2}g\left(R\left(u,\left.\frac{\partial}{\partial x^i}\right|_x\right)\left.\frac{\partial}{\partial x^l}\right|_x,R\left(u,\left.\frac{\partial}{\partial x^i}\right|_x\right)\left.\frac{\partial}{\partial x^l}\right|_x\right) \\
     = & \frac{1}{2}\sum_{l=1}^{n}\sum_{s \neq i, t \neq i}u^su^tg\left(R\left(\left.\frac{\partial}{\partial x^s}\right|_x,\left.\frac{\partial}{\partial x^i}\right|_x\right)\left.\frac{\partial}{\partial x^l}\right|_x,R\left(\left.\frac{\partial}{\partial x^t}\right|_x,\left.\frac{\partial}{\partial x^i}\right|_x\right)\left.\frac{\partial}{\partial x^l}\right|_x\right).
  \end{split}
\end{equation*}
Hence
$$D_i=\frac{1}{2}\sum_{l=1}^{n}\sum_{s \neq i, t \neq i,r}u^su^t u^rg\left(R\left(\left.\frac{\partial}{\partial x^s}\right|_x,\left.\frac{\partial}{\partial x^r}\right|_x\right)\left.\frac{\partial}{\partial x^l}\right|_x,R\left(\left.\frac{\partial}{\partial x^t}\right|_x,\left.\frac{\partial}{\partial x^r}\right|_x\right)\left.\frac{\partial}{\partial x^l}\right|_x\right)+f_i,$$
with $f_i$ is a $C^{\infty}$ function on $M_x$ which does not depend on $u^i$.

For $i \neq j,$ deriving  equation \eqref{DD} twice with respect to $u_j$ and $u_i$ and summing up, we obtain
$$\frac{\partial^2}{\partial^2 u^j}\left(\frac{\partial D_i}{\partial u^i}\right)+\frac{\partial^2}{\partial^2 u^i}\left(\frac{\partial D_j}{\partial  u^j}\right)=\sum_{l=1}^{n}\frac{\partial }{\partial u^i}\left(\frac{\partial}{\partial u^j}\left(g\left(R\left(u,\left.\frac{\partial}{\partial x^i}\right|_x\right)\left.\frac{\partial}{\partial x^l}\right|_x,R\left(u,\left.\frac{\partial}{\partial x^j}\right|_x\right)\left.\frac{\partial}{\partial x^l}\right|_x\right)\right)\right),$$
which gives
$$2\sum_{l=1}^{n}\left\|R\left(\frac{\partial}{\partial x^i},\frac{\partial}{\partial x^j}\right)\frac{\partial}{\partial x^l}\right\|^2=-\sum_{l=1}^{n}\left\|R\left(\frac{\partial}{\partial x^i},\frac{\partial}{\partial x^j}\right)\frac{\partial}{\partial x^l}\right\|^2.$$
Hence the Riemannian curvature $R$ is zero, i.e. $ (M, g) $ is flat. We deduce that $(TM,G)$ is flat and $Z$ is a homothetic vector field, with homothety factor $\bar{\lambda}$.

The converse is easy to prove.
\end{Proof}

To classify completely Ricci solitons for $G=ag^s+bg^h+cg^v$, $a \neq 0$, $a+c \neq 0$ and $a(a+c) -b^2 \neq 0$, it is sufficient to characterize homothetic vector field on $(TM,G)$. We start with the following lemma which gives a local characterization of the Lie derivative with respect to an arbitrary vector field on $TM$:

\begin{Lem}\label{Lie-Z}
	Let $(M,g)$ be a flat Riemannian manifold and $G=ag^s+bg^h+cg^v$. Then for all $Z \in \mathfrak{X}(TM)$ and $X, Y \in \mathfrak{X}(M)$, we have
	\begin{equation*}
	\begin{split}
	\mathcal{L}_ZG(X^h,Y^h)= & [(a+c) A^i +bB^i] \mathcal{L}_{\frac{\partial}{\partial x^i}}g(X,Y)\\
&+[(a+c)X^h(A^i) +bX^h(B^i)]g(\frac{\partial}{\partial x^i},Y)\\
&+[(a+c)Y^h(A^i) +bY^h(B^i)]g(X,\frac{\partial}{\partial x^i}) ,\\ \\
	\mathcal{L}_ZG(X^h,Y^v)= &[b^i A^i +aB^i] g(\nabla_{\frac{\partial}{\partial x^i}}X,Y)\\
&+[bX^h(A^i) +aX^h(B^i)]g(\frac{\partial}{\partial x^i},Y)\\
&+[(a+c)Y^v(A^i) +bY^v(B^i)]g(X,\frac{\partial}{\partial x^i}) ,\\ \\
	\mathcal{L}_ZG(X^v,Y^v)= & [bX^v(A^i) +aX^v(B^i)]g(\frac{\partial}{\partial x^i},Y)\\
&[bY^v(A^i) +aY^v(B^i)]g(X,\frac{\partial}{\partial x^i}),
	\end{split}
	\end{equation*}
where $Z=A^i\left(\frac{\partial}{\partial x^i}\right)^h +B^i\left(\frac{\partial}{\partial x^i}\right)^v$.
\end{Lem}

For a vector field $\xi$ on $M$, we denote by $C(\xi)$ the $(1,1)$-tensor field on $M$ defined by $g(C(\xi)Y,Z)=-g(Y,\nabla_Z \xi)$, for all vector fields $Y$ and $Z$ on $M$. Locally, if $\xi =\xi^i  \frac{\partial}{\partial x^i}$, then $C(\xi)^i_j=-g^{ik}g_{jl}\xi^l_{;k}$, where$\xi^l_{;k}$ are the local components of the $(1,1)$-tensor field $\nabla \xi$. We denote also by $I$ the identity $(1,1)$-tensor field on $M$. Then we have

\begin{Th}\label{hom-flat}
	Let $(TM, G)$ be the tangent bundle of a flat Riemannian manifold $(M,g)$ with $G=ag^s+bg^h+cg^v,$ and $Z$ be a vector field on $TM.$ $Z$ is a homothetic vector field on $(TM,G)$ if and only if $Z$ is expressed as
\begin{equation}\label{Z}
\begin{split}
  Z= & \frac1\alpha \{[a\zeta -b\xi]^h + *[a(C(\xi) -bP +\lambda ab I]\\
    & +[(a+c)\xi -b\zeta]^v  +\iota[(a+c)P -bC(\xi) +\lambda (\alpha -b^2)I]\},
\end{split}
\end{equation}
where
	\begin{enumerate}
		\item $\zeta$ is a homothetic vector field on $M$ satisfying $\mathcal{L}_{\zeta}g=2\lambda (a+c)g$;
		\item $P$ is a $(1,1)$ tensor field on $M$ which is skew-symmetric and parallel;
		\item $\xi$ is a vector field on $(M,g)$  satisfying  $\nabla^2\xi(U,V)+\nabla^2\xi(V,U)=0$ for any $U,V\in \mathfrak{X}(M).$
	\end{enumerate}
\end{Th}
\begin{Proof}
	Suppose that $Z$ is a homothetic vector field, then there is a constant $\lambda$ such that
	\begin{equation}\label{C1}
	\mathcal{L}_ZG=2\lambda G.
	\end{equation}
	Taking, in the third equation of Lemma \ref{Lie-Z}, $X=\frac{\partial}{\partial x^i}$ and $Y=\frac{\partial}{\partial x^j}$ and taking into account \eqref{C1}, we obtain
    \begin{equation}\label{C2}
      2\lambda a g_{ij}=b\left[\frac{\partial A_i}{\partial u^j} +\frac{\partial A_j}{\partial u^i}\right] +a\left[\frac{\partial B_i}{\partial u^j} +\frac{\partial B_j}{\partial u^i}\right], \quad i,j=1,...,n,
    \end{equation}
    where $A_i:=g_{ik}A^k$ and $B_i:=g_{ik}B^k$. Putting
    \begin{equation}\label{C3}
      W_i:=bA_i+aB_i -\lambda a u_i,
    \end{equation}
    where $u_i:=g_{ik}u^k$, it is easy to see that \eqref{C2} is equivalent to
    \begin{equation}\label{C4}
      \frac{\partial W_i}{\partial u^j} +\frac{\partial W_j}{\partial u^i}=0, \quad i,j=1,...,n.
    \end{equation}
    For $i=j$, \eqref{C4} gives
    \begin{equation}\label{C5}
      \frac{\partial W_i}{\partial u^i}=0, \quad i=1,...,n.
    \end{equation}
    Deriving \eqref{C4} with respect to $u^j$ and taking into account \eqref{C5}, we obtain $\frac{\partial^2 w_i}{(u^j)^2}=0$, for all $i,j=1,...,n$. It follows that there are functions $\xi_i$, $P_{ij}$, $i,j=1,...,n$, such that
    \begin{equation}\label{C6}
      W_i=\xi_i +P_{ik}u^k, \quad i=1,...,n.
    \end{equation}
    Equation \eqref{C4} is equivalent to
    \begin{equation}\label{C7}
      P_{ij}+P_{ji}=0, \quad i,j=1,...,n.
    \end{equation}
    Now we put
    \begin{equation}\label{C8}
      V^i:=(a+c)A^i +bB^i, \quad i=1,...,n.
    \end{equation}
    Taking, in the second equation of Lemma \ref{Lie-Z}, $X=\frac{\partial}{\partial x^i}$ and $Y=\frac{\partial}{\partial x^j}$ and taking into account \eqref{C1}, \eqref{C6} and \eqref{C8}, we obtain
    \begin{equation}\label{C9}
      2\lambda bg_{ij} =  g_{jk}\xi^k_{;i} +g_{jk}P^k_{l;i}u^l +g_{ik}\frac{\partial V^k}{\partial u^j}, \quad i,j=1,...,n,
    \end{equation}
    where $\xi^i:=g^{ik}\xi_k$, $P^i_j:= g^{ik}P_{kj}$ and $\xi^i_{;j}$ and $P^i_{j;k}$ are the local components of the differentials $\nabla \xi$ and $\nabla P$ of the local vector field $\xi:=\xi^i \frac{\partial}{\partial x^i}$ and the local $(1,1)$-tensor field $P:=P^i_j \frac{\partial}{\partial x^i} \otimes dx^j$, respectively, given by
    \begin{equation*}
      \xi^i_{;j}=\frac{\partial \xi^i}{\partial x^j} +\Gamma^i_{jk} \xi^k, \quad P^i_{j;k}= \frac{\partial \xi^k}{P^i_j} +\Gamma^i_{kl}P^l_j -\Gamma^l_{kj}P^i_l.
    \end{equation*}
    It follows then, from \eqref{C9}, that
    \begin{equation}\label{C10}
      \frac{\partial V^i}{\partial u^j}=2\lambda b \delta^i_j -g^{il}g_{jk}\xi^k_{;l} -g^{im}g_{jk}P^k_{l;m}u^l, \quad i,j=1,...,n.
    \end{equation}
    We deduce that
    \begin{equation}\label{C11}
      V^i= \zeta^i+ 2\lambda bu^i -g^{il}g_{jk}\xi^k_{;l}u^j +T^i_{jk} u^ju^k,
    \end{equation}
    where $\zeta^i$ and $T^i_{jk}$ are $C^\infty$-functions such that $T^i_{jk}$ are symmetric in $j$ and $k$. Substituting from \eqref{C11} into \eqref{C10}, we obtain
    \begin{equation*}
      2T^i_{jk} = -g^{im}g_{jl}P^l_{k;m}.
    \end{equation*}
    But we deduce from \eqref{C7} that $-g^{im}g_{jl}P^l_{k;m}$ is skew-symmetric, and hence
    \begin{equation}\label{C12}
    P^i_{j;k}=0,\quad i,j,k=1,...,n.
    \end{equation}
    It follows then that
    \begin{equation*}
      V^i= \zeta^i+ 2\lambda bu^i -g^{il}g_{jk}\xi^k_{;l}u^j, \quad i=1,...,n,
    \end{equation*}
    or, in other words,
    \begin{equation}\label{C13}
      V^i= \zeta^i+ 2\lambda bu^i +C(\xi)^i_ju^j, \quad i=1,...,n,
    \end{equation}
    where $C(\xi)$ is the $(1,1)$-tensor field given in Theorem \eqref{hom-flat}, given locally by $C(\xi)^i_j=-g^{il}g_{jk}\xi^k_{;l}$.

    Finally, taking, in the first equation of Lemma \ref{Lie-Z}, $X=\frac{\partial}{\partial x^i}$ and $Y=\frac{\partial}{\partial x^j}$ and taking into account \eqref{C1} and \eqref{C13}, we obtain
    \begin{equation}\label{C14}
      \begin{split}
        2\lambda (a+c) g_{ij} & = g_{ik}\zeta^k_j +g_{jk}\zeta^k_i +(g_{ik}C(\xi)^k_{l;j} +g_{jk}C(\xi)^k_{l;i})u^l \\
          & = g_{ik}\zeta^k_{;j} +g_{jk}\zeta^k_{;i} -g_{lm}(\xi^m_{;ij} +\xi^m_{;ji})u^l, \quad i,j=1,...,n.
      \end{split}
    \end{equation}
    We deduce that
    \begin{equation}\label{C15}
      2\lambda (a+c) g_{ij} = g_{ik}\zeta^k_{;j} +g_{jk}\zeta^k_{;i}, \quad i,j=1,...,n,
    \end{equation}
    i.e. $\zeta$ is a homothetic vector field of homothety factor $\lambda (a+c)$, and
    \begin{equation}\label{C16}
      \xi^m_{;ij} +\xi^m_{;ji}=0, \quad i,j=1,...,n,
    \end{equation}
    i.e. $\nabla^2 \xi$ is skew-symmetric.

    Now, substituting from \eqref{C6} and \eqref{C13} into \eqref{C3} and \eqref{C8}, we obtain
	\begin{equation}\label{C17}
	\left\{
	\begin{split}
	A^i= & \frac1\alpha \{a\zeta^i -b\xi^i +[aC(\xi)^i_j -bP^i_j +\lambda ab I^i_j]u^j\},\\
	B^i= & \frac1\alpha \{-b\zeta^i +(a+c)\xi^i +[-bC(\xi)^i_j +(a+c)P^i_j +\lambda (\alpha-b^2) I^i_j]u^j\},
	\end{split}
	\right.
	\end{equation}
$i=1,...,n$. We deduce that $Z$ is expressed as \eqref{Z}.

The converse is easy to prove.	
\end{Proof}


\section{Ricci solitons of Kaluza-Klein type on unit tangent sphere bundles}



{ In this section, we suppose that $(M,g)$ is a Riemannian manifold of constant sectional curvature and that its unit tangent sphere bundle $T_1M$ is endowed with a $g$-natural metric $\tilde{G}$ of Kaluza-Klein type.
	We now  characterize vector fields $V$ on $T_1M$, which, together with a metric $\widetilde{G}$ of Kaluza-Klein type, give rise to a Ricci soliton. 
A routine calculation, using Propositions \ref{cor-LC} and \ref{cor-ric}, yields the following	
	
	\begin{Pro}\label{RSeqG}
		Let $(M,g)$ be a Riemannian manifold of constant sectional curvature $\kappa$, $\widetilde{G}$ be a pseudo-Riemannian $g$-natural metric of Kaluza-Klein type on $T_1M$, $V$ be a vector field on $T_1 M$ and $\lambda \in \mathbb{R}$. Then $(\widetilde{G},V,\lambda)$ is a Ricci soliton if and only if the following equations are satisfied:
		\begin{equation}\label{GC3}
	\left\{
	\begin{array}{ll}
	\left(\mathcal{L}_{V}\widetilde{G}\right)_{(x,u)}(X^h,Y^h) =& 2(a+c)(\lambda-\nu) g(X,Y)+2(\lambda d-\theta)g(X,u)g(Y,u),\\[4pt]
	\left(\mathcal{L}_{V}\widetilde{G}\right)_{(x,u)}(X^h,Y^t) =& 0,\\[4pt]
	\left(\mathcal{L}_{V}\widetilde{G}\right)_{(x,u)}(X^t,Y^t) =& 2a[\lambda-\mu]g(X,Y),
	\end{array}
	\right.
	\end{equation}
		for all $x \in M$, $(x,u) \in T_1 M$ and $X$, $Y \in M_x$ satisfying Convention \ref{conv}, where
	\begin{equation}\label{mu-nu}
	\left\{
	\begin{array}{ll}
	\mu= & \frac{1}{2a(a+c)\varphi} [a^2\kappa^2 +2(n-2)(a+c)\varphi -d^2], \\[4pt]
	\nu= & \frac{1}{2a(a+c)\varphi}  [-a^2\kappa^2 +2(n-1)a\varphi \kappa +d(d-2\varphi)], \\[4pt]
	\theta= & \frac{1}{2a(a+c)\varphi}[-a^2((n-2)\varphi +d)\kappa^2  +d(2n(a+c)\varphi+(n-1)d\varphi -d(a+c))],
	\end{array}
	\right.
	\end{equation}
	or, in other words
	\begin{equation*}
	\mathcal{L}_{V}\widetilde{G}=2a[\lambda-\mu]\widetilde{g^s}+2[a(\mu-\nu)+c(\lambda-\nu)]\widetilde{g^v}+2[\lambda d-\theta]\widetilde{k^v},
	\end{equation*}
where $\widetilde{g^s}$, $\widetilde{g^v}$ and $\widetilde{k^v}$ are the induced metrics on $T_1M$ from $g^s$, $g^v$ and $k^v$, respectively.
	\end{Pro}

	\begin{Rk}
		With the same hypotheses of Proposition \ref{RSeqG}, if  $(\widetilde{G},V,\lambda)$ is a Ricci soliton then
		\begin{itemize}
			\item $V$ is a conformal vector field on $(T_1M,\widetilde{G})$ if and only if $\theta=\mu d$ and $\mu=\nu,$
			\item $V$ is a Killing vector field on $(T_1M,\widetilde{G})$ if and only if $\theta=\lambda d$ and $\lambda=\mu=\nu$.
		\end{itemize}
	\end{Rk}


\subsection{Case when the potential vector field is a complete lift vector field }
	
	
	Recall that if $\xi$ is a vector field on $M$, then its complete lift $\xi^{\bar{c}}$ to $T_1M$ can be expressed as $\xi^{\bar{c}}=X^h +[\iota(\nabla \xi)]^t$. Using Lemmas \ref{Lie-tang} and \ref{Lie-hor}, we have
	
	\begin{Lem}\label{Lie-complete} \cite{KA}
		For all $\xi \in \mathfrak{X}(M)$, $(x,u) \in T_1M$ and $X, Y \in T_xM$, we have
		\begin{equation*}
		\begin{split}
		\left(\mathcal{L}_{\xi^{\bar{c}}}\tilde{G}\right)_{(x,u)}(X^h,Y^h)= & (a+c)\left(\mathcal{L}_{\xi}g\right)_x(X,Y) +d\{\left(\mathcal{L}_{\xi}g\right)_x(Y,u)g(X,u)\\
		&+\left(\mathcal{L}_{\xi}g\right)_x(X,u)g(Y,u)-\left(\mathcal{L}_{\xi}g\right)_x(u,u)g(Y,u)g(X,u), \\
		&  \\
		\left(\mathcal{L}_{\xi^{\bar{c}}}\tilde{G}\right)_{(x,u)}(X^h,Y^t)= & ag(R(\xi_x,X)u+\nabla^2\xi(u,X),Y),\\ \\
		\left(\mathcal{L}_{\xi^{\bar{c}}}\tilde{G}\right)_{(x,u)}(X^t,Y^t)= & a\left[\left(\mathcal{L}_{\xi}g\right)_x(X,Y)-\left(\mathcal{L}_{\xi}g\right)_x(u,u)g(X,Y)\right].
		\end{split}
		\end{equation*}
	\end{Lem}
	
	\begin{Th}\label{Ccomplet}
		Let $(M,g)$ be a Riemannian manifold of constant sectional curvature $\kappa$, $\widetilde{G}$ be a pseudo-Riemannian $g$-natural metric of Kaluza-Klein type on $T_1M$, $\xi$ be a vector field on $M$ and $\lambda \in \mathbb{R}$. Then, $(\widetilde{G},\xi^{\bar{c}},\lambda)$ is a Ricci soliton if and only if the two following assertions holds
        \begin{enumerate}
          \item $\xi$ is a homothetic vector field on $M$, with $\mathcal{L}_{\xi}g=2\lambda_0g$, where $\lambda_0 \in \mathbb{R}$;
          \item one of the following cases occurs
		\begin{itemize}
			\item [i)] $n=2,d=0,$ $\lambda=\frac{a\kappa^2}{2(a+c)^2}$  and $\lambda_0=\kappa\frac{a\kappa-(a+c)}{(a+c)^2}$;
			\item [ii)] $n=2, d\neq 0,$ $\kappa=\frac{2\varphi-d}{a}$ and $\lambda=\lambda_0=\frac{2}{a}$;
			\item [iii)] $n\neq 2, d=0 $, $\kappa=0$ and $\lambda_0=\lambda=\frac{n-2}{a}$;
			\item  [iv)]$n\neq 2,$ $d\neq 0,$ $\kappa=-\frac{1}{a(n-2)}[(n-1)d\pm\sqrt{d^2+2n(n-2)\varphi d}]$, $\lambda=\mu$ and $\lambda_0=\mu-\nu$, where $\mu$ and $\nu$ are given by \eqref{mu-nu}.
		\end{itemize}
        \end{enumerate}
	\end{Th}
	\begin{proof}
		Suppose that $(\widetilde{G},\xi^{\bar{c}},\lambda)$ is a Ricci soliton. \\For all $(x,u) \in TM $ and $X,Y\in M_x.$ Substituting from the first equation of Lemma \ref{Lie-complete} into the first equation of  \eqref{GC3} and taking $X = Y \perp u $, we obtain
		\begin{equation}
		g(\nabla_X\xi,X)=(\lambda-\nu)g(X,X).
		\end{equation}
		Then $\xi$ is a homothetic vector field with $\mathcal{L}_{\xi}g=2(\lambda-\nu)g$. \\ Substituting from the third equation of Lemma \ref{Lie-complete} into \eqref{GC3}, we obtain
		\begin{equation}\label{lam1}
		\lambda=\mu=\frac{1}{2a(a+c)\varphi} [a^2\kappa^2 +2(n-2)(a+c)\varphi -d^2].
		\end{equation}
        We deduce that $\xi$ is a homothetic vector field, i.e. $\mathcal{L}_{\xi}g=2\lambda_0 g$, with
        \begin{equation}\label{lam0}
          \lambda_0=\mu-\nu.
        \end{equation}
		Substituting from the first equation of Lemma \ref{Lie-complete} into the first equation of  \eqref{GC3}, we obtain
		\begin{equation}\label{theta}
		\nu d=\theta.
		\end{equation}
		Thus
		\begin{equation*}
		(n-2)a^2\varphi \kappa^2 +2(n-1)ad\varphi \kappa +d[d^2-(n+1)d\varphi+(a+c)(d-2n\varphi)]=0,
		\end{equation*}
		i.e,
		\begin{equation}\label{eq1}
		(n-2)a^2 \kappa^2 +2(n-1)ad \kappa -nd(2\varphi-d)=0.
		\end{equation}
		Then we can distinguish the following situations:\\
        \begin{itemize}
          \item [i)] $n=2$ and $d=0$: substituting into \eqref{mu-nu}, we get $\mu=\frac{a}{2}(\frac{\kappa}{a+c})^2$ and $\nu=\frac{\kappa}{2(a+c)^2}[-a\kappa+2(a+c)]$.
Then we have, by virtue of \eqref{lam1} and \eqref{lam0}, $\lambda=\frac{a\kappa}{2(a+c)^2}$ and $\lambda_0=\frac{\kappa}{(a+c)^2}[a\kappa -(a+c)]$.
          \item [ii)] $n=2$ and $d\neq 0$: \eqref{eq1} yields $\kappa=\frac{2\varphi-d}{a}$. Substituting into \eqref{mu-nu}, we get $\mu=\frac{2}{a}$ and $\nu=0$. Replacing in \eqref{lam1} and \eqref{lam0}, we obtain $\lambda=\lambda_0=\frac{2}{a}$.
          \item [iii)] $n \neq 2$ and $d=0$: then we have, from \eqref{theta}, $\theta=0$. Using \eqref{mu-nu}, we obtain $\kappa=0$ and thus substituting into \eqref{mu-nu}, we get $\mu=(n-2) \frac{a+c}{a\varphi}$ and $\nu=0$. Replacing in \eqref{lam1} and \eqref{lam0}, we obtain $\lambda=\lambda_0=\frac{n-2}{a}$.
          \item [iv)] $n \neq 2$ and $d\neq 0$: from equation \eqref{eq1} we obtain
		\begin{equation}
		[\kappa+\frac{(n-1)d}{(n-2)a}]^2=\frac{d}{(n-2)^2a^2}[d^2+2n(n-2)\varphi d].
		\end{equation}
		We deduce that $\kappa=-\frac{1}{a(n-2)}[(n-1)d\pm\sqrt{d^2+2n(n-2)\varphi d}]$.
        \end{itemize}

Conversely, by a routine calculation, we can check that in any case $i)$-$iv)$ of 2. of the theorem, we get
\begin{equation}\label{cond1}
  \lambda=\mu, \quad \lambda -\nu= \lambda_0, \quad \lambda d- \theta= \lambda_0 d.
\end{equation}
On the other hand, it is well known that, for any homothetic vector field $\xi$ on $M$, we have
\begin{equation}\label{hom-curv}
  R(\xi,X)Y+\nabla^2\xi(Y,X)=0,
\end{equation}
for any $X, Y \in \mathfrak{X}(M)$. Substituting from the identity $\mathcal{L}_{\xi}g=2\lambda_0g$ and \eqref{hom-curv} into Lemma \ref{Lie-complete}, and using \eqref{cond1}, we obtain
\begin{equation*}
		\begin{split}
		\left(\mathcal{L}_{\xi^{\bar{c}}}\tilde{G}\right)_{(x,u)}(X^h,Y^h)= & 2\lambda_0(a+c)g(X,Y) +2\lambda_0dg(Y,u)g(X,u)\\
		=&2(a+c)(\lambda-\nu) g(X,Y)+2(\lambda d-\theta)g(X,u)g(Y,u), \\
		&  \\
		\left(\mathcal{L}_{\xi^{\bar{c}}}\tilde{G}\right)_{(x,u)}(X^h,Y^t)= & 0,\\ \\
		\left(\mathcal{L}_{\xi^{\bar{c}}}\tilde{G}\right)_{(x,u)}(X^t,Y^t)= & 0= 2a[\lambda-\mu]g(X,Y),
		\end{split}
		\end{equation*}
which is no other than \eqref{GC3}, i.e. $(\widetilde{G},\xi^{\bar{c}},\lambda)$ is a Ricci soliton.
	\end{proof}
\begin{Rks}
Let $(M,g)$ be a Riemannian manifold of constant sectional curvature $\kappa$, $\tilde{G}$ be a pseudo-Riemannian $g$-natural metric on $T_1M$ of Kaluza-Klein type, $\xi$ be a vector field on $M$ and $\lambda\in \mathbb{R} $.
	\begin{enumerate}		
		\item $(T_1M,\tilde{G})$ is an Einstein manifold if and only if $\xi$ is a Killing vector field on $M$, i.e. $\lambda_0=0$.
		\item For $n=2$ and $d=0$, $(\tilde{G},\xi^{\bar{c}},\lambda)$ is shrinking, steady or expanding according to $a< 0$, $\kappa=0$ or $a> 0$, respectively.
		\item For  $n= 2$  and $d\neq 0$, $(\tilde{G},\xi^{\bar{c}},\lambda)$ is shrinking or expanding according to $a< 0$ or $a> 0$, respectively.
        \item For  $n\neq 2$  and $d= 0$, $(\tilde{G},\xi^{\bar{c}},\lambda)$ is shrinking or expanding according to $a< 0$ or $a> 0$, respectively.
		\item For $n\neq 2$ and $d\neq 0$, $(\tilde{G},\xi^{\bar{c}},\lambda)$ is shrinking, steady, or expanding according to $\mu$ is negative, zero or positive, respectively.
		\item In the case 2. iv) of Theorem \ref{Ccomplet}, we should have $d^2+2n(n-2)\varphi d \geq 0$, i.e. one of the following situations occurs
		\begin{itemize}
			\item $a+c< 0$ and $d\in ]-\infty,0] \cup [\frac{-2n(n-2)(a+c)}{2n(n-2)+1},+\infty[$,
			\item $a+c> 0$ and $d\in ]-\infty, \frac{-2n(n-2)(a+c)}{2n(n-2)+1}] \cup [0,+\infty[$.
		\end{itemize}
	\end{enumerate}
\end{Rks}


\subsection{Necessary conditions on potential vector fields}


	 Since $T_1M$ is a closed submanifold of $TM$, then for any vector field $V$ on $T_1M$, there is a vector field $\overline{V}$ on $TM$ which extends $V$. For a local coordinates system $(U,x^1,...,x^n)$ of $M$, $\overline{V}$ can be expressed in the induced coordinates system of $TM$ as
	\begin{equation*}
	\begin{array}{l}
	\overline{V}\upharpoonright_{p^{-1}(U)}=A^i \left(\frac{\partial}{\partial x^i}\right)^h +B^i \left(\frac{\partial}{\partial x^i}\right)^v,
	\end{array}
	\end{equation*}
	where $A^i$ and $B^i$ are smooth functions on $p^{-1}(U)$. Since $\overline{V}$ is tangent to $T_1M$ at any point of $T_1M$,  we should have $g_{ij}(x)B^i(x,u)u^j=0$, for all $(x,u) \in p_1^{-1}(U)=p^{-1}(U) \cap T_1M$ expressed locally as $u =u^i\frac{\partial}{\partial x^i}$, i.e.,
	\begin{equation}\label{GC1}
	B^i(x,u)u_i=0,
	\end{equation}
	where we put $u_i:=g_{ij}(x)u^j$. If we denote the restrictions to $p_1^{-1}(U)$ of $A^i$ and $B^i$ by the same notations, then we can write
	\begin{equation}\label{GC2}
	\begin{array}{l}
	V\upharpoonright_{p_1^{-1}(U)}=A^i \left(\frac{\partial}{\partial x^i}\right)^h +B^i \left(\frac{\partial}{\partial x^i}\right)^t.
	\end{array}
	\end{equation}
	Using Lemmas \ref{Lie-tang} and \ref{Lie-hor}, we can prove the following
	\begin{Lem}\label{gen-case}
		For all $(x, u)\in T_1M, $ $X, Y\in M_x, $ we have
		\begin{equation*}
	\begin{split}
	\left(\mathcal{L}_{V}\widetilde{G}\right)_{(x,u)}(X^h,Y^h)= &(a+c)\left[A^i \left(\mathcal{L}_{\frac{\partial}{\partial x^i}}g\right)_x(X,Y) \right.\\ &\left.+X^h(A^i)g(\left.\frac{\partial}{\partial x^i}\right|_x,Y)+Y^h(A^i)g(X,\left.\frac{\partial}{\partial x^i}\right|_x)\right]\\
	&+d\left\{B^i\left[g(X,\left.\frac{\partial}{\partial x^i}\right|_x)g(Y,u)+g(\left.\frac{\partial}{\partial x^i}\right|_x,Y)g(X,u)\right]\right.\\
	& +A^i\left[g(\nabla_X \frac{\partial}{\partial x^i},u)g(Y,u)+g(\nabla_Y \frac{\partial}{\partial x^i},u)g(X,u)\right]\\
& \left.+X^h(A^i)g(\left.\frac{\partial}{\partial x^i}\right|_x,u)g(Y,u)+Y^h(A^i)g(\left.\frac{\partial}{\partial x^i}\right|_x,u)g(X,u)\right\}, \\
	& \\
	\left(\mathcal{L}_{V}\widetilde{G}\right)_{(x,u)}(X^h,Y^t)= & a\left[X^h(B^i)g(\left.\frac{\partial}{\partial x^i}\right|_x,Y)+B^ig(\nabla_{X}\frac{\partial}{\partial x^i},Y)\right.\\
    & \left.+kA^i\left(g(X,u)g(\left.\frac{\partial}{\partial x^i}\right|_x,Y)-g(X,Y)g(\left.\frac{\partial}{\partial x^i}\right|_x,u)\right)\right]\\
    &+(a+c)Y^t(A^i)g(X,\left.\frac{\partial}{\partial x^i}\right|_x)+dY^t(A^i)g(X,u)g(\left.\frac{\partial}{\partial x^i}\right|_x,u),\\
	& \\
	\left(\mathcal{L}_{V}\widetilde{G}\right)_{(x,u)}(X^t,Y^t)= &  a\left[Y^t(B^i)g(X,\left.\frac{\partial}{\partial x^i}\right|_x)+X^t(B^i)g(Y,\left.\frac{\partial}{\partial x^i}\right|_x)\right].
	\end{split}
	\end{equation*}
		\end{Lem}
	We put $W=V-A^l(\frac{\partial}{\partial x^l})^h$. $W$ is a fiber-preserving vector field, and we have, from Lemma \ref{gen-case}, $\mathcal{L}_{W}\widetilde{G}(X^t,Y^t)= \mathcal{L}_{V}\widetilde{G}(X^t,Y^t)$. Using \eqref{GC3}, we deduce that
	\begin{equation}
	\mathcal{L}_{W}\widetilde{G}(X^t,Y^t) =2[\lambda-\mu]\tilde{G}(X^t,Y^t),
	\end{equation}
	In the same way, we remark that
    \begin{equation}\label{Lie-Sasaki}
      \mathcal{L}_{W}\widetilde{g}^s(X^t,Y^t) =2[\lambda-\mu]\tilde{g}^s(X^t,Y^t).
    \end{equation}

	We shall extend fiber-preserving vector fields on $T_1M$ to vector fields on $TM,$ using the technics used in \cite{Kon1}. Let $Z$ be a fiber-preserving vector field on $T_1M.$ Then it is known that $Z$ is projectable to a vector field $\underline{Z}$ on $M,$ i.e. such that $(dp_1)_u(Z_u) = \underline{Z}_x,$ for all $x\in M $ and $u\in S_xM. $ For all $r > 0,$ let us define the immersions $j_r : T_1M\rightarrow TM$ and $j_0 : M \rightarrow TM,$ respectively, by $j_r(u) = ru,$ for all $u\in T_1M, $ and $j_0(x) = 0_x,$ for all $x\in M, $ where $0_x$ denotes the zero vector in $M_x.$ We define a vector field $\overline{\overline{Z}}$ on $TM$ extending $Z$ by (cf. \cite{Kon1})
	\begin{equation*}
	\overline{\overline{Z}}_{ru}:=\left\{\begin{array}{ll}
	(dj_r)_u(Z_u), & \quad \textup{ for } r>0, \\
	(dj_0)_x(\underline{Z}_x), &\quad \textup{ for } r=0,
	\end{array}
	\right.
	\end{equation*}
for all $x \in M$ and $u \in S_xM:=M_x \cup T_1M$. We denote by $\overline{Z}$ the restriction of $\overline{\overline{Z}}$ to $TM\setminus \sigma_0,$ which is clearly a vector field on $TM\setminus \sigma_0,$, where $\sigma_0 := j_0(M)$ is the zero section of $TM.$
\begin{Lem}\label{Lem-con1}\cite{KA}
	Let $Z$, $X$ and $Y$ be fiber-preserving vector fields on $T_1M$. Then
	\begin{equation}\label{rel-Lie}
	\left(\mathcal{L}_{\overline{Z}}g^s(\overline{X},\overline{Y})\right)_{ru}=(1-r^2)(a+c)\left(\mathcal{L}_{\underline{Z}}g(\underline{X},\underline{Y})\right)_{x} +r^2\left(\mathcal{L}_Z\tilde{g}^s(X,Y)\right)_u,
	\end{equation}
	for all $x \in M$, $u \in S_xM$ and $r>0$.
\end{Lem}
Now, fixing $u \in S_xM$  and taking, in (\ref{rel-Lie}),  $Z=\overline{W}$ and $X=Y=T^t$, for $T \in \mathfrak{X}(M)$ such that $0 \neq T_x \perp u$, and using Lemma \ref{Lem-con1}, we obtain
\begin{equation*}
\mathcal{L}_{\overline{W}_{ru}}g^s(T^v_{ru},T^v_{ru})=\mathcal{L}_{Z_u}\widetilde{g^s}(T^v_u,T^v_u).
\end{equation*}
Since, by \eqref{Lie-Sasaki}, we have $\mathcal{L}_{Z_u}\widetilde{g^s}(T^v_u,T^v_u)=2[\lambda-\mu]g(T_x,T_x)$, then
\begin{equation*}
\mathcal{L}_{\overline{W}_{ru}}g^s(T^v_{ru},T^v_{ru})=2[\lambda-\mu]g(T_x,T_x),
\end{equation*}
for all $r>0$. When $r \rightarrow 0$, we have by continuity
\begin{equation*}
\mathcal{L}_{\overline{\overline{Z}}_{0}}g^s(T^v_{0},T^v_{0})=2[\lambda-\mu]g(T_x,T_x).
\end{equation*}
But since $\overline{\overline{Z}}_{0}=0$, then $2[\lambda-\mu]g(T_x,T_x)=0$, and consequently
\begin{equation}
\lambda=\mu=\frac{1}{2a(a+c)\varphi} [a^2\kappa^2 +2(n-2)(a+c)\varphi -d^2]
\end{equation}
and $\mathcal{L}_{\overline{\overline{Z}}}g^s(T^v_{ru},T^v_{ru})=0$. Using the third identity of Lemma \ref{Lie-Z} for $G=g^s$, we obtain easily, as in the proof of Theorem \ref{hom-flat},
\begin{equation}\label{B-loc}
  B^i=Q_{j}^{i}u^j+\zeta^i, \quad i=1,...,n,
\end{equation}
where $Q_{i}^{j}, \zeta^l$ are smooth functions such that
\begin{equation}\label{sym-Q}
  g_{ik}Q_{j}^{k}+g_{jk}Q_{i}^{k}=0, \quad i,j=1,...,n.
\end{equation}
Moreover since $B^iu_i=0$, then $Q_{j}^{i}u_ju^i+\zeta^iu_i=0$. Replacing $u$ by $-u$ in the last identity, we deduce that $\zeta^i=0$ for all $i=1,...,n,$ and \eqref{B-loc} becomes $B^i=Q_{j}^{i}u^j$. Considering the $(1,1)$-tensor field $Q$ on $M$ whose components are $Q^{i}_{j}$, then we have $W_{(x,u)}=Q(u)$, for all $(x,u) \in T_1M$, or equivalently $t\{V\}=\tilde{\iota} (Q)$, with $Q$ is skew-symmetric. This completes the proof of the following:
\begin{Pro}\label{tang-V}
 Let $(M,g)$ be a Riemannian manifold of constant sectional curvature $\kappa$, $\tilde{G}$ be a psuedo-Riemannian $g$-natural metric on $T_1M$ of Kaluza-Klein type, $V$ be a vector field on $T_1M$ and $\lambda\in \mathbb{R} $. If $(V,\tilde{G},\lambda)$ is a Ricci soliton on $T_1M$ then we have
 \begin{enumerate}
 	\item $\lambda=\frac{1}{2a(a+c)\varphi} [a^2\kappa^2 +2(n-2)(a+c)\varphi -d^2]$;
 	\item $t\{V\}=\tilde{\iota} (Q)$, where $Q$ is a skew-symmetric $(1,1)$-tensor field.
 \end{enumerate}
\end{Pro}


\subsection{Case when the potential vector field is fiber-preserving}


In this section, we shall give necessary and sufficient conditions for the unit tangent bundle of a constant curvature Riemannian manifold, endowed with a pseudo-Riemannian Kaluza-Klein type metric, to be a Ricci soliton with a fiber preserving potential vector field.

\begin{Th}
   Let $(M,g)$ be an $n$-dimensional Riemannian manifold of constant sectional curvature $\kappa$, $\tilde{G}$ be a pseudo-Riemannian $g$-natural metric on $T_1M$ of Kaluza-Klein type, $V$ be a fiber-preserving vector field on $T_1M$ and $\lambda \in \mathbb{R}$. Then $(\tilde{G},V,\lambda)$ is a  Ricci soliton on $T_1M$ if and only if there are:
   \begin{enumerate}
     \item A conformal vector field $\xi$ on $M$, with $\mathcal{L}_{\xi}g=2 \lambda_0 g$;
     \item A parallel $(1,1)$-tensor field $P$ on $M$, with $P -(\lambda-\nu) I$ is skew-symmetric;
   \end{enumerate}
   such that one of the following cases occurs:
   \begin{itemize}
			\item [i)] $n=2$, $d=0$, $\lambda=\frac{ak^2}{2(a+c)^2}$, $\lambda_0=k\frac{ak-(a+c)}{(a+c)^2}$ and $V=\xi^{\bar{c}}+\tilde{i}P$;
			\item [ii)] $n=2, d\neq 0,$ $k=\frac{2\varphi-d}{a}$, $\lambda=\lambda_0=\frac{2}{a}$ and $V=\xi^{\bar{c}}$;
			\item [iii)] $n\neq 2$, $d=0 $, $k=0$, $\lambda_0=\lambda=\frac{n-2}{a}$ and $V=\xi^{\bar{c}}+\tilde{i}P$;
			\item  [iv)]$n\neq 2,$ $d\neq 0,$ $k=-\frac{1}{a(n-2)}[(n-1)d\pm\sqrt{d^2+2n(n-2)\varphi d}]$, $\lambda=\mu$, $\lambda_0=\mu-\nu$ and $V=\xi^{\bar{c}}$;
		\end{itemize}
where $I$ is the identity $(1,1)$-tensor field on $M$ and $\mu$ and $\nu$ are given by \eqref{mu-nu}.
   \end{Th}

\begin{Proof}
Let $V$ be a fiber-preserving vector field on $T_1M$. Then
\begin{equation}
h\{[V,X^t](x,u)\}=0,
\end{equation}
 for all $(x,u)\in T_1M $ and $X\in M_x,$ where $h$ stands for horizontal component. Using the identity $B^i(x,u)u_i = 0$, we obtain $h\{[V,X^t](x,u)\} = h\{−X^t(A^l)\frac{\partial}{\partial_l} \}$.  Then $h\{[V,X^t](x,u)\} =0$ if and only if
 \begin{equation}
 X^t(A^l)=0,
 \end{equation}
 for all $(x,u\in T_1M $ and $X\in M_x$. It follows that, for $X\perp u,$, we have  $(dA^l)_{(x,u)}(X^{t}_{(x,u)}) = 0$, for all $X\perp u.$ If we denote by $A^{l}_{x} $ the restriction of $A^l$ to the fiber $S_xM := T_1M\cap M_x$, we then have
 \begin{equation}\label{FBC}
 (dA^{l}_{x})_{(x,u)}(X^{t}_{(x,u)}) = (dA^l)_{(x,u)}(X^{t}) = 0,
 \end{equation}
   for all $X\in M_x $ such that $g(X,u) = 0$. But $dA^{l}_{x}$ is a linear form defined on the tangent space $(S_xM)_{(x,u)} = \{X^{t }_{(x,u)}/ X \in M_x, g(X,u) = 0\}.$ Thus, by \eqref{FBC} it follows that
  $(dA^{l}_{x})_{(x,u)}$ vanishes identically on $(S_xM)_{(x,u)}$. We conclude that the restriction $A^{l}_{x}$ is constant on $S_xM$. Therefore, for any $x \in M,$ there is a $C^{\infty}$-function $\xi^l$ on $M$, such that
  \begin{equation}
  A^{l}_{(x,u)} = \xi^l(x).
  \end{equation}
  i.e, $\pi_*(V)=\xi$, where $\xi$ is the vector field on $M$ given locally by $\xi=\sum_{i=1}^{n}\xi^i(\frac{\partial}{\partial x^i})$. It follows then that $h\{V\}=\xi^h$. But, by Proposition \ref{tang-V}, we have $t\{V\}=\tilde{\iota} Q$, for some skew symmetric $(1,1)$-tensor field $Q$ on $M$. We deduce that $V$ is given by $V = \xi^h+\tilde{\iota} Q$.  It is easy to see from  Lemma \ref{gen-case} that, for all $(x,u)\in  T_1M,$  $X,Y\in  M_x,$ we have
  \begin{equation}\label{420}
  \left\{
  \begin{array}{ll}
  \left(\mathcal{L}_{V}\widetilde{G}\right)_{(x,u)}(X^h,Y^h) =& (a+c)\left(\mathcal{L}_{\xi}g\right)_x(X,Y)\\
                                         & +d[g(\nabla_X\xi,u)g(Y,u)+g(\nabla_Y\xi,u)g(X,u)\\
                                         & +g(Q(u),X)g(Y,u)+g(Q(u),Y)g(X,u)];\\[4pt]
  \left(\mathcal{L}_{V}\widetilde{G}\right)_{(x,u)}(X^h,Y^t) =& ag(R(\xi_x,X)u+\nabla_X Q(u),Y);\\[4pt]
  \left(\mathcal{L}_{V}\widetilde{G}\right)_{(x,u)}(X^t,Y^t) =& a[g(Q(X),Y)+g(Q(Y),X)].
  \end{array}
  \right.
  \end{equation}
  Putting, in the first equation of \eqref{420}, $X=Y=u$ and using \eqref{GC3}, we get
  \begin{equation}
  \mathcal{L}_{\xi}g(X,X)=2(\lambda-\nu)\|X\|^2,
  \end{equation}
  for all $X \perp u$. Since $u$ is arbitrary, the previous identity holds for all $X \in TM$. Then, by bilinearity, we get
  \begin{equation}\label{homot1}
  \mathcal{L}_{\xi}g(X,Y)=2(\lambda-\nu)g(X,Y),
  \end{equation}
  for all $x \in M$ and $X, Y \in M_x$.
  Taking $X = u$ and $Y \perp u$ into the first equation of \eqref{420} we obtain, by virtue of \eqref{GC3} and \eqref{homot1},
  \begin{equation}\label{prs}
  d\left[g(Q(u),Y)+g(\nabla_Y\xi,u)\right]=0,
  \end{equation}
  for all $Y \perp u$. So,we have two possibilities:\\
  \textbf{Case 1:} If $d\neq 0$, then from \eqref{prs} we obtain $g(Q(u),Y)=-g(\nabla_Y\xi,u)$, for all $Y \perp u$. But we have, by \eqref{homot1},  $g(\nabla_Y\xi,u)+g(\nabla_Y\xi,u)=2(\lambda-\nu)g(Y,u)=0$, for all $Y \perp u$. It follows that
  \begin{equation}\label{L1}
  g(Q(u)-\nabla_u\xi,Y)=0,
  \end{equation}
 for all $Y\perp u$. Now, since $Q$ is skew-symmetric, then we have, by virtue of \eqref{homot1},
 \begin{equation}\label{L2}
g(Q(u)-\nabla_u\xi,u)=-g(\nabla_u\xi,u)=(\nu-\lambda),
 \end{equation}
 for all $(x,u)\in T_1M$. Combining \eqref{L1} and \eqref{L2}, we obtain $Q(u) =\nabla_u\xi+(\nu-\lambda)u,$ and consequently $(\tilde{\iota}Q)_{(x,u)}=(\nabla_u\xi)^t_{(x,u)},$ since $u^t_{(x,u)}=0$. We deduce then that $V_{(x,u)}=\xi^{h}_{(x,u)}+(\nabla_u\xi)^t=\xi^{\bar{c}}_{(x,u)}$, i.e.
 \begin{equation}
 V=\xi^{\bar{c}}.
 \end{equation}
 \textbf{Case 2:} If $d=0$, then we consider the $(1,1)$ tensor field $P$ given by $P(X)=Q(X)-\nabla_X\xi$, for all $X \in TM$. Then $V$ can be expressed as $V=\xi^{\bar{c}}+\tilde{i}P$. Recall that, by \eqref{homot1}, $\xi$ is homothetic and hence \eqref{hom-curv} holds. Then the system \eqref{420} reduces to
  	\begin{equation}\label{GCd0}
  \left\{
  \begin{array}{ll}
  \left(\mathcal{L}_{V}\widetilde{G}\right)_{(x,u)}(X^h,Y^h) =& 2(a+c)(\lambda-\nu)g_x(X,Y);\\[4pt]
  \left(\mathcal{L}_{V}\widetilde{G}\right)_{(x,u)}(X^h,Y^t) =& ag(\nabla_XP(u),Y);\\[4pt]
  \left(\mathcal{L}_{V}\widetilde{G}\right)_{(x,u)}(X^t,Y^t) =& a[g(P(X),Y)+g(P(Y),X)].
  \end{array}
  \right.
  \end{equation}
 Comparing with \eqref{GC3}, we obtain $\theta=0$, $\nabla P=0$ and $g(P(X),Y)+g(P(Y),X)=2(\lambda-\nu)g(X,Y)$. Now $\theta=0$ means by \eqref{mu-nu} that either $n=2$ or $k=0$.

 The converse is straightforward.
\end{Proof}

\begin{flushleft}
Laboratory of Algebra, Geometry and Arithmetic,\\
Department of Mathematics,\\
Faculty of sciences Dhar El Mahraz,\\
University Sidi Mohamed Ben Abdallah,\\
B.P. 1796, Fes-Atlas,
Fez, Morocco.\\
E-mail address: mtk abbassi@Yahoo.fr; amri.noura1992@gmail.com.
\end{flushleft}

\begin{thebibliography}{xx}
	
	{
\bibitem{KA} M.T.K. Abbassi, N. Amri,
    \emph{On conformal vector field on unit tangent sphere bundles with $g$-natural metrics},
    to appear in Czech. Math. J.
	
\bibitem{KAC} M.T.K. Abbassi, N. Amri, G. Calvaruso,
    \emph{Kaluza-Klein type Ricci Solitons on Unit Tangent Sphere Bundles},
    Diff. Geom. Appl. \textbf{59} (2018), 184-203.
		
\bibitem{Abb-Cal1} M.T.K. Abbassi, G. Calvaruso,
	\emph{$g$-natural contact metrics on unit tangent sphere bundles},
    Monaths. Math.  \textbf{151} (2007), 89--109.
		
\bibitem{Abb-Cal4} M.T.K. Abbassi, G. Calvaruso,
	\emph{$g$-natural metrics of constant curvature on unit tangent sphere bundles},
		Arch. Math. (Brno) \textbf{48} (2012), 81--95.
		
\bibitem{ACP1} M.T.K. Abbassi, G. Calvaruso, D. Perrone,
    \emph{Harmonic sections of tangent bundles equipped with Riemannian $g$-natural metrics},
    Quart. J. Math., {\bf 62} (2011), 259--288.
		
\bibitem{Abb-Kow1} M.T.K. Abbassi,  O. Kowalski,
	\emph{On $g$-natural metrics with constant scalar curvature on unit tangent sphere bundles},
    Topics in Almost Hermitian Geometry and related fields, Proc. of the Int. Conf. in Honor of K. Sekigawa's 60th birthday, World Scientific, 2005, 1--29.
		
\bibitem{Abb-Kow2} M.T.K. Abbassi, O. Kowalski,
	\emph{Naturality of homogeneous metrics on Stiefel manifolds $SO(m+1)/SO(m-1)$},
    Diff. Geom. Appl., \textbf{28} (2010), 131--139.
		
\bibitem{Abb-Kow3}K.M.T. Abbassi, O. Kowalski,
	\emph{On Einstein Riemannian $g$-natural metrics on unit tangent sphere bundles},
    Ann. Global. Anal. Geom. \textbf{38} (2010), 11--20.
		
\bibitem{Abb-Sar4} M.T.K. Abbassi, M. Sarih,
	\emph{On some hereditary properties of Riemannian $g$-natural metrics on tangent bundles of Riemannian manifolds},
    Diff. Geom. Appl. \textbf{22} (2005), 19--47.

\bibitem{Abb-Sar5} M.T.K. Abbassi, M. Sarih,
	\emph{On Riemannian $g$-natural metrics of the form $a.g^s + b.g^h + c.g^v$ on the tangent bundle of a Riemannian manifold $(M, g)$},
    Mediterr. J. Math. \textbf{2} (2005), 19--43.
		
		
\bibitem{BGR} M. Brozos-Vazquez, G. Calvaruso, E. Garcia-Rio, S. Gavino-Fernandez,
    {\em Three-dimensional Lorentzian homogeneous  Ricci solitons},
    Israel J. Math., {\bf 188} (2012), 385--403.
		
\bibitem{CMM} G. Calvaruso, V. Martin-Molina,
    \emph{Paracontact metric structures on the unit tangent sphere bundle},
    Ann. Mat. Pura Appl., {\bf 194} (2015), 1359--1380.
		
\bibitem{CP} G. Calvaruso, D. Perrone,
    \emph{Homogeneous and $H$-contact unit tangent sphere bundles},
    Austral. J. Math., {\bf 88} (2010),  323--337.
		
\bibitem{CP2} G. Calvaruso, D. Perrone,
    \emph{Geometry of Kaluza-Klein metrics on the sphere $\mathbb S^3$},
    Ann. Mat. Pura Appl., {\bf 192} (2013), 879--900.
		
\bibitem{CP3} G. Calvaruso, D. Perrone,
    \emph{Metrics of Kaluza-Klein type on the anti-de Sitter space $\mathbb{H}_1^3$},
    Math. Nachr., {\bf 287} (2014), 885--902.
		
\bibitem{Cao} H. D. Cao,
    \emph{Geometry of Ricci solitons},
    Chinese Ann. Math. Ser. B \textbf{27B} (2006), 121-–142.
		
\bibitem{Cho-Kno} B. Chow, D. Knopf,
    \emph{The Ricci Flow: An Introduction,}
	Mathematical Surveys and Monographs, \textbf{110}. American Mathematical Society, Providence, 2004.
		
\bibitem{Hall} G. Hall,
    \emph{Symmetries and curvature structure in general relativity},
    World Sci. Lect. Notes in Physics, vol. \textbf{46}, 2004.
		
\bibitem{Kol-Mic-Slo} I. Kol\'a\v r, P.W. Michor, J. Slov\'ak,
	\emph{Natural operations in differential geometry},
	Springer-Verlag, Berlin, 1993.
		
\bibitem{Kow-Sek1} O. Kowalski, M. Sekizawa,
	\emph{Natural transformations of Riemannian metrics on manifolds to metrics on tangent bundles -a classification-},
    Bull. Tokyo Gakugei Univ. \textbf{40} (1988), 1--29.
		
\bibitem{Per} D. Perrone,
	\emph{Geodesic Ricci solitons on unit tangent sphere bundles},
    Ann. Global. Anal .Geom. 44 (2013), 91--103.
		
\bibitem{Kon1} T. Konno,
    \emph{Killing vector fields on tangent sphere bundles},
    Kodai Math. J. {\bf 21} (1998), 61--72.
		
\bibitem{Wood} C.M. Wood,
    \emph{An existence theorem for harmonic sections},
    Manuscripta Math., {\bf 68} (1990), 69--75.
	}
	
\end{thebibliography}
\end{document}